\documentclass[reqno,11pt,oneside]{amsart}

\usepackage{amsthm,amsmath}
\RequirePackage[numbers]{natbib}
\usepackage{amssymb}
\usepackage{graphicx,subfigure,epstopdf}
\usepackage{pstricks,pst-node,pst-tree, pst-plot, pst-eps}
\usepackage{color}
\usepackage{enumerate}
\RequirePackage[colorlinks,citecolor=blue,urlcolor=blue]{hyperref}

\newcommand{\al}{\alpha}
\newcommand{\fy}{\varphi}

\usepackage{nameref}

\makeatletter
\let\orgdescriptionlabel\descriptionlabel
\renewcommand*{\descriptionlabel}[1]{%
  \let\orglabel\label
  \let\label\@gobble
  \phantomsection
  \edef\@currentlabel{#1}%
  \let\label\orglabel
  \orgdescriptionlabel{#1}%
}
\makeatother

\theoremstyle{plain}
\newtheorem{theorem}{Theorem}[section]

\theoremstyle{definition}
\newtheorem{definition}[theorem]{Definition}
\newtheorem{proposition}[theorem]{Proposition}
\newtheorem{remark}[theorem]{Remark}
\newtheorem{corollary}[theorem]{Corollary}
\newtheorem{lemma}[theorem]{Lemma}
\newtheorem{example}[theorem]{Example}
\numberwithin{equation}{section}

\def\al{\alpha}

\def\fy{\varphi}
\def\Om{\Omega}

\def\Gd{D_{\infty}^{(\rho)}}

\def\Mitaga1{{E_{\al,1}}}

\def\L2o{{L_2(\Om)}}

\def\II{\Omega}

\def\Om{\Omega}
\def\II{(\Om)}

\newcount\icount

\begin{document}

%\begin{frontmatter}

% "Title of the paper"

%\runtitle{Stochastic representation for nonlocal-in-time diffusion}

% indicate corresponding author with \corref{}
% \author{\fnms{John} \snm{Smith}\corref{}\ead[label=e1]{smith@foo.com}\thanksref{t1}}
% \thankstext{t1}{Thanks to somebody}
% \address{line 1\\ line 2\\ printead{e1}}
% \affiliation{Some University}

\title[\hfil Stochastic representation of solution to nonlocal-in-time diffusion\MakeLowercase{s}\hfil]{{\itshape Stochastic representation of solution to nonlocal-in-time diffusion}}
\author{\textsc{Q\MakeLowercase{iang} D\MakeLowercase{u}, L\MakeLowercase{orenzo} T\MakeLowercase{oniazzi}} \MakeLowercase{and} \textsc{Z\MakeLowercase{hi} Z\MakeLowercase{hou}}}
\date{\today}
\thanks{QD (Columbia University) is supported in part by  NSF DMS-1719699,
AFOSR  MURI center for Material Failure Prediction through peridynamics and  the ARO MURI Grant W911NF-15-1-0562.\\
LT (University of Warwick) is supported by the EPSRC, UK. \\
ZZ (The Hong Kong Polytechnic University) is partially supported by the start-up grant from the Hong Kong Polytechnic University and
 a grant from the Research Grants Council of the Hong Kong Special Administrative Region (Project No. 25300818). }

\address{Qiang Du \\ Department of Applied Physics and Applied Mathematics\\
 Columbia University\\
  New York 10027, USA
	}
  \email{qd2125@columbia.edu}
	
	\address{Lorenzo Toniazzi\newline
Department of Mathematics,\\
 University of Warwick,\\
 Coventry, United Kingdom. }
\email{l.toniazzi@warwick.ac.uk} 

\address{Zhi Zhou\\
Department of Applied Mathematics\\
The Hong Kong Polytechnic University\\
 Kowloon, Hong Kong, China}
 \email{zhizhou@polyu.edu.hk}
%\author{\fnms{Qiang} \snm{Du}\thanksref{m1}\ead[label=e1]{}},

%\author{Lorenzo Toniazzi\thanksref{m2}\ead[label=e2]{l.toniazzi@warwick.ac.uk}}

%\author{Zhi Zhou\corref{}\thanksref{m3}\ead[label=e3]{}}

%\and
%\author{\fnms{???} \snm{???}\ead[label=e2]{???}}
%\address{\printead{e2}}
%\affiliation{???}

\maketitle

%Columbia University, University of Warwick  and The Hong Kong Polytechnic University

%\thankstext{t1}{
%\runauthor{Du, Toniazzi, Zhou}

\begin{abstract}
The aim of this paper is to give a stochastic representation for the solution to a natural extension of the Caputo-type evolution equation. The nonlocal-in-time operator  is defined by a hypersingular integral with a (possibly time-dependent) kernel function, and it results in a model which serves a bridge between normal diffusion and anomalous diffusion.
%After reformulating the original problem into a Caputo-type nonlocal evolution equation with a specific forcing term, we derive the stochastic representation of the solution, subject to certain smoothness and compatibility conditions on both initial and source data.
We derive the stochastic representation for the weak solution of the nonlocal-in-time problem in case of nonsmooth data.
We do so by starting from an auxiliary  Caputo-type evolution equation with a specific forcing term.
Numerical simulations are also provided to support our theoretical results.
\end{abstract}

%\end{frontmatter}

% AOS,AOAS: If there are supplements please fill:
%\begin{supplement}[id=suppA]
%  \sname{Supplement A}
%  \stitle{Title}
%  \slink[doi]{10.1214/00-AOASXXXXSUPP}
%  \sdatatype{.pdf}"
%  \sdescription{Some text}
%\end{supplement}

\section{Introduction}\label{section:intro}
In this paper, we study the   nonlocal-in-time evolution equation
\begin{equation}\label{preRL}
\left\{
\begin{split}
  D_{\infty}^{(\rho)}  u(t,x) - \Delta   u(t,x)&=  f(t,x), &  &~  ~ (t,x)\in(0,T]\times \Omega,\\
  u(t,x)&=0,   & & ~ ~(t,x)\in(0,T]\times \partial\Omega, \\
  u(t,x)&=\phi(t,x), &  &~ ~(t,x)\in (-\infty,0]\times \Omega,
\end{split}
\right.
\end{equation}
where $\Omega\subset \mathbb R^d$ is a regular domain, the functions $f$ and $\phi$ are given data, and $ D_{\infty}^{(\rho)}$ denotes the nonlocal operator defined by
\begin{align}\label{eqn:op}
  D_{\infty}^{(\rho)} u(t)   &:= \int_0^{\infty} (u(t) - u(t-r))\rho(t,r)\,dr, %& t\in(-\infty,T],
\end{align}
with the nonnegative kernel function $\rho \ge 0$ satisfying certain hypothesis (see details in Section \ref{sec:prelim}).
The nonlocal operator $ -D_{\infty}^{(\rho)}$ is proved to be the Markovian generator of a $(-\infty,T]$-valued decreasing L\'evy-type  process, denoted by $-X^{t,(\rho)}$ when started at $t\in[0,T]$.
We denote by $B^x$ a $d$-dimensional Brownian motion started at $x\in \mathbb R^d$ generated by the Laplacian $\Delta$.
The processes $-X^{t,(\rho)}$ and $B^x$ are always assumed to be independent.

%{\color{red}physical background}

%Studying stochastic representations for solutions to PDE's it has proved fruitful in fractional calculus, theoretically and in applications.
The aim of the current work is to derive a stochastic representation for the solution to the problem \eqref{preRL} with the historical initial condition.
Besides their theoretical importance, stochastic representations are extensively used in applications, e.g.,  to compute solutions through the particle tracking method (see, e.g., \cite{ZM08, ZM06}).
It is a deep and classical result that the solution to the diffusion equation
\begin{equation*}
\left\{
\begin{split}
  \partial_t u(t,x) &= \Delta   u(t,x),&  & ~ ~(t,x)\in (0,T]\times \mathbb{R}^d,\\
  u(0,x)&=\phi(0,x),&  &~ ~x\in \mathbb{R}^d,
\end{split}
\right.
\end{equation*}
allows the stochastic representation $u(t,x)=\mathbf E[\phi(0,B^x(t))]$. This normal diffusion model describes diffusion phenomena that exhibits homogeneity in both space and time.
With the aid of single particle tracking,
recent studies have provided many examples of
anomalous diffusion.
One typical example is the time-fractional (sub-)diffusion model,
\begin{equation} \label{eqn:frac}
\left\{
\begin{split}
  \partial_t^\alpha u(t,x) &= \Delta   u(t,x), &  &~ ~ (t,x)\in (0,T]\times \mathbb{R}^d,\\
  u(0,x)&=\phi(0,x),&  &~ ~ x \in \mathbb{R}^d,
\end{split}
\right.
\end{equation}
where $\partial_t^\alpha$  denotes the  Caputo fractional derivative of order $\alpha\in(0,1)$, which can be defined by
\begin{equation*}
  \partial_t^\alpha u (t)=  \int_0^t  \frac{(t-r)^{-\alpha}}{\Gamma(1-\alpha)} \partial_r u(r)  \,dr.
\end{equation*}
The sub-diffusion phenomena has attracted much attention in applications such as contaminant transport in groundwater \cite{kirchner2000fractal}, protein diffusion within cells \cite{golding2006physical}, and thermal diffusion in fractal media \cite{nigmatullin1986realization}.
 The problem \eqref{eqn:frac} has been extensively studied  both analytically and numerically (see \cite[Chapter 2.4]{Meerschaert2012} for an overview).
Its solution can be expressed by $u(t,x)=\mathbf E[\phi(0,Y^x(t))]$ \cite{schef04},
where  $Y^x(t)= B^x(\tau_0^\alpha(t))$ and $\tau_0^\alpha(t)=\inf \{s>0:X^\alpha(s)\ge t\}$
is the inverse process of the $\alpha$-stable subordinator $X^\alpha$.
The density of  $Y^x(t)$ can be derived using a conditioning argument \cite{Zal97, BM01}
\begin{equation}\label{FSC}
H_{t,x}(y)=\int_0^\infty p_s(x,y)\partial_s\mathbf P[X^{\alpha}(s)\ge t]\,ds,
\end{equation}
where $\partial_s\mathbf P[X^{\alpha}(s)\ge t]=\alpha^{-1}ts^{-1- 1/\alpha}g_\alpha(ts^{-1/\alpha})$,
with $g_\alpha$ being the density of $X^\alpha(1)$ and $p_s(x)$ the density of $B^x(s)$. %This pretty formula was found in \cite{Zal97}, and proved for much more general spatial operators in \cite{BM01}.
It is interesting to observe that the time-changed Brownian motion $Y^x(t)$ displays time heterogeneity,
 as the non-Markovian time change $t\mapsto \tau_0^\alpha(t)$ is constant precisely when the subordinator $t\mapsto X^\alpha(t)$ jumps  \cite{Meerschaert2012}.  This leads to the past-dependent diffusion  $Y^x$ being trapped, and in general spreading
 at a slower rate than $B^x$ (see e.g. \cite{Zas94, Sai05, Mag07}). Let us recall that $Y^x$ is sometimes called fractional kinetic and it enjoys surprising universality properties \cite{barlow2011convergence}.
 Moreover, the result can be generalized to other Caputo-type derivatives \cite{MS06, Vel11, Chen17, HKT17}.
It is easy to see that the Caputo fractional derivative can be written in the form \eqref{eqn:op} by
\begin{equation*}
  \partial_t^\alpha u (t)=  c_\alpha \int_0^\infty (u(t)-u(t-r)) r^{-\alpha-1} \,dr,
\end{equation*}
with the kernel $\rho(t,r):=c_\alpha r^{-\alpha-1}$, $c_\alpha =-\Gamma(-\alpha)^{-1}$, where we extend the function $u$ to the negative real line by $u(t)\equiv u(0)$ for $t\in(-\infty,0)$.
On the other hand, under certain hypothesis,
one may show that the nonlocal operator could reproduce the first order derivative,
as the horizon of nonlocal effects tends to zero \cite{DYZ17}.
Therefore, it is actually an interesting intermediate case between infinite-horizon fractional derivatives and infinitesimal local derivatives.
Moreover, it can be shown that the nonlocal setting also serves to bridge between a short-time anomalous diffusion and a long-time normal diffusion \cite{DZ2018},
which has been observed in many experiments \cite{He2016}. %The model with time-independent kernel has been analytically studied in \cite{DYZ17} by Galerkin's method.

Compared with the fractional diffusion model \eqref{eqn:frac}, the nonlocal-in-time model \eqref{preRL} requires a historical initial data, which could be time-dependent. As far as we know, the only work concerning the stochastic explanation  of the historical initial data is \cite{T18}, which deals with the fractional case. In this work, we derive a stochastic representation of the solution to the problem \eqref{preRL} with a possibly time-dependent kernel $\rho$ and a historical initial data $\phi$.  As an example, we prove that the weak solution to the homogenous problem (for $f=0$) allows the stochastic representation
\begin{equation}\label{SRin0}
\begin{split}
  u(t,x)&= \mathbf E\left[ \phi\left(-X^{t,(\rho)}(\tau_0(t)),B^x(\tau_0(t))\right)\mathbf 1_{\{\tau_0(t)< \tau_\Omega(x)\}}\right]\\
  &=\int_{-\infty}^0\int_\Omega \phi(r,y)H_{t,x}(r,y)\,dr\,dy,
\end{split}
\end{equation}
where $\tau_0(t)=\inf \{s>0:-X^{t,(\rho)}(s)\le 0\}$, $\tau_\Omega(x)= \inf \{s>0:x+B(s)\notin\Omega\}$
 and the heat kernel is given by
\[
H_{t,x}(r,y)=\int_0^t \rho(z,z-r)\left(\int_0^\infty p^\Omega_s(x,y)\partial_z\mathbf P[-X^{t,(\rho)}(s) \le z]\,ds\right)dz.
\]
Here we denote by $p^\Omega_s(x,y)$ the density of the killed Brownian motion $B^x(s)\mathbf 1_{\{s< \tau_\Omega(x)\}}$. Note that for the standard fractional kernel $\rho(t,r)=c_\alpha r^{-\alpha-1}$, $-X^{t,(\rho)}=t-X^\alpha$ and $\tau_0(t)=\tau_0^\alpha(t)$.
The representation \eqref{SRin0} appears to be new, and it suggests an interesting interpretation.
This is because the diffusion on $\Omega$ is still the anomalous diffusion $Y^x(t)=B^x (\tau_0(t) )$, but the
contribution in time of the initial condition $\phi(\cdot,Y^x(t))$ depends on the waiting/trapping time of $Y^x(t)$,
which is indeed $W(t)=X^{t,(\rho)} (\tau_0(t))$. Let us stress that as a particular case we treat Caputo-type EEs.

The paper is organized as follows. In Section \ref{sec:prelim}, we introduce some basic settings of the nonlocal-in-time model
\eqref{preRL} as well as some probabilistic background. Some popular and concrete models, which satisfy certain hypothesis, will be provided as examples. In Section \ref{sec:generaltheory}, after reformulating the model \eqref{preRL} into a Caputo-type fractional diffusion problem,
we develop some general solution theory,
provided additional smoothness and compatibility conditions on problem data.
In Section \ref{sec:weak} we prove Theorem \ref{thm:main}, where we show that the candidate stochastic representation provides a weak solution of \eqref{preRL} even
though the data is weak.
Finally, we present some numerical experiments to illustrate our theoretical findings.
Throughout, the notation $c$ denotes a generic positive constant, whose value may differ at each occurrence.

%To see why the representation is natural ....

\section{Preliminaries}\label{sec:prelim}

\subsection{General notation}
We denote by $\mathbb{N}$, $\mathbb R^+$, $\mathbb{R}^d$, $a\wedge b$,  $\Gamma (\cdot)$, $\mathbf 1_E(\cdot)$ and a.e., the set of positive integers, the set of non-negative real numbers,  the $d$-dimensional Euclidean space, the minimum between $a,b\in\mathbb R$, the Gamma function, the indicator function of the set $E$ and the statement almost everywhere with respect to Lebesgue measure, respectively.  To ease notation, $F(I)=FI$ whenever $F(I)$ is a space of real-valued functions on an interval $I\subset \mathbb R$. We denote by $\|\cdot\|_B$ the norm of a Banach space $B$, and if $\mathcal L$ is a bounded linear operator between Banach spaces, we denote its operator norm by $\|L\|$. We denote by $C(E)$ the space of real-valued continuous functions on $E\subset\mathbb R^d$. For any  set $E\subset \mathbb R^d$, any bounded open set $\Omega\subset \mathbb R^d$ and $T\ge0$,  we define the Banach spaces
\begin{align*}
B(E)=&\ \{f: E\to \mathbb R:\text{ $f$ is bounded and measurable}\},\\
C_{\infty}(E)=&\ \{f \in C(E):\, f \text{ vanishes at infinity}\}.\\
C_0[0,T]=&\ \{f\in C[0,T]: f(0)=0\}, \\
C_{\partial\Omega}([0,T]\times \Omega)=&\ \{f \in C([0,T]\times \overline\Omega): f=0\text{ on } [0,T]\times\partial\Omega\},\\
C_{0,\partial\Omega}([0,T]\times \Omega)=&\ \{f \in C_{\partial\Omega}([0,T]\times \Omega): f=0\text{ on } \{0\}\times\overline\Omega\},\\
C_{\infty,\partial\Omega}((-\infty,  T]\times \Omega)=&\ \{f\in C_{\infty}((-\infty, T]\times \overline\Omega): f=0\text{ on } \partial\Omega\},
\label{eq:}
\end{align*}
all equipped with the supremum norm, which we often denote by $\|\cdot\|_\infty$.    For an open set $\Omega\subset\mathbb R^d$ we define
\begin{align*}	
C^k(\Omega)=&\ \{f\in C(\Omega): D^\gamma f\in C(\Omega),\  |\gamma|\le k\},\\
C^\infty(\overline\Omega)=&\ \{f\in C^k(\Omega): D^\gamma f\in C(\overline\Omega),\  |\gamma|\le k, \, k\in\mathbb N\},\\
%C^k(\overline\Omega)=&\ \{f\in C^k(\Omega): \ D^\gamma f \text{ is uniformly continuous on }\Omega,\  |\gamma|\le k\},\\
%C^\infty(\overline\Omega)=&\ \cap_{k=1}^\infty C^k(\overline\Omega),\\
C_c^\infty(\Omega) =&\ \{f\in C^\infty(\overline\Omega): f\text{ vanishes outside of a compact set }K\subset\Omega \},\\
C^1_0[0,T]=&\ \{f\in C^1[0,T]: f,f'\in C_0[0,T]\},
\end{align*}
where $\gamma$ is a multi-index, $D^\gamma$ the associated integer-order derivative operator. %,  for the  space $C^k(\overline\Omega)$ we identify the functions on $\Omega$ along with their partial derivatives with their unique continuous extension to $\overline\Omega$.
We denote by $L^p(\Omega)$, $p\in[1,\infty]$ the usual Banach spaces of Lebesgue $p$-integrable real-valued functions on $\Omega$.  We define $L^p(a,b;B)=\{f:(a,b)\to B\text{ such that }t\mapsto \|f(t)\|_B\in L^p(a,b)\}$, for $p\in[1,\infty]$ and $b>a\ge-\infty$. If $F$ and $\tilde F$ are two sets of real-valued functions, we define $F\tilde F :=\{f\tilde f:\ f\in F,\ \tilde f\in\tilde F\}$, and we denote by $\text{Span}\{F\}$ the set of all linear combination of elements in $F$.

 The notation we use for an $E $-valued stochastic process started at $x\in E $ is $X^x=\{X^x(s)\}_{s\ge0}$. Note that the symbol $t$ will often be used to denote the starting point of a stochastic process with state space $E \subset \mathbb R$. By a \emph{ strongly continuous contraction semigroup} $T$ we mean a collection of linear operators $T_s:B\to B$, $s\ge0$, where $B$ is a Banach space, such that $T_{s+r}=T_sT_r$, for every $s,r\ge 0$, $T_0$ is the identity operator, $\lim_{s\downarrow 0}T_sf=f$ in $B$, for every $f\in B$,  and $\sup_s\|T_s\|\le1$. The generator of the semigroup $T$ is defined as the pair $(\mathcal L,\text{Dom}(\mathcal L))$, where $\text{Dom}(\mathcal L):=\{f\in B: \mathcal L f:=\lim_{s\downarrow0}s^{-1}(T_sf-f) \text{ exists in }B\}$. We say that a set $C\subset \text{Dom}(\mathcal L)$ is a \emph{core for} $(\mathcal L,\text{Dom}(\mathcal L))$ if the generator equals the closure of the restriction of $\mathcal L$ to $C$. We say that a set $C\subset B$ is \emph{invariant under} $T$ if $T_sC\subset C$ for every $s\ge0$. If a set $C$ is invariant under $T$ and a core for  $(\mathcal L,\text{Dom}(\mathcal L))$, then we say that $C$ is an \emph{invariant core for} $(\mathcal L,\text{Dom}(\mathcal L))$. Recall that if $C$ is a dense subspace of $\text{Dom}(\mathcal L)$ and $C$ is invariant under $T$, then $C$ is an invariant core for $(\mathcal L,\text{Dom}(\mathcal L))$ (see \cite[Lemma 1.34]{Schilling}). For a given $\lambda\ge0$ we define the \emph{resolvent of} $T$  by $(\lambda-\mathcal L)^{-1}:=\int_0^\infty e^{-\lambda s}T_s \,ds$, and recall that for $\lambda >0$, $(\lambda-\mathcal L)^{-1}:B\to \text{Dom}(\mathcal L)$ is a bijection and it solves the abstract resolvent equation
\begin{equation*}
\mathcal L(\lambda-\mathcal L)^{-1}f=\lambda (\lambda-\mathcal L)^{-1}f-f,\quad f\in B,
\end{equation*}
 see for example \cite[Theorem 1.1]{Dyn65}. By a \emph{Feller semigroup} we mean a strongly continuous contraction semigroup $T$ on any of the (compactified) Banach spaces of continuous functions defined above such that $T$ preserves non-negative functions. 
% and we call any such triplet $(P,B,(\mathcal L,\text{Dom}(\mathcal L)))$ a \emph{Feller triplet}.
 A Feller semigroup $T$ is said to be \emph{conservative} if the  extension of $T$ to bounded measurable functions preserves constants. Feller semigroups are in one-to-one correspondence with Feller processes, where a Feller process is a time-homogenous sub-Markov process $\{X(s)\}_{s\ge0}$ such that $s\mapsto T_sf(x):=\mathbf E[f(X(s))|X(0)=x]$, $f\in B$ is a Feller semigroup \cite[Chapter 1.2]{Schilling}. We recall that every Feller process admits a c\'adl\'ag modification which enjoys the strong Markov property \cite[Theorem 1.19 and Theorem 1.20]{Schilling}, and we always work with such modification. 
For further discussions on these terminologies and notations, we refer to \cite{Schilling}.

\subsection{Nonlocal operators and  related stochastic processes}\label{section2.3}
Next, we review some basics on the nonlocal operators, along with some properties and related definitions.
\begin{description}
\item[(H0)\label{H0}] The function $\rho: \mathbb{R} \times (0,\infty) \to [0,\infty)$ is  continuous and continuously  differentiable in the first variable.    Furthermore,
  \begin{equation}\nonumber
  \int_0^{\infty}(1\wedge r ) \sup_t\rho (t,r)\,dr < \infty, \quad \int_0^{\infty} (1\wedge r ) \sup_t \Big | \partial_t \rho(t,r)\Big |\,dr < \infty,
  \end{equation}
  and
  \begin{equation}\nonumber   \lim_{\delta \to 0} \sup_t \int_{0 < r \le \delta} r \rho(t,r)\,dr = 0.
   \end{equation}
Moreover, there exist  $\epsilon> 0$ and  $\gamma > 0$, such that  the function $\rho$ satisfies $ \rho (t,r) \ge \gamma > 0$ for all $t$ and $|r| < \epsilon$.
\end{description}

\begin{definition}
For any kernel function $\rho$ satisfying condition \ref{H0},  the \textit{ Marchaud-type derivative} $D_{\infty}^{(\rho)}$ and the \textit{ Caputo-type derivative}  $D^{(\rho)}_0$ are respectively defined  by
\begin{align}
  D_{\infty}^{(\rho)} u(t)   &:= \int_0^{\infty} ( u(t) - u(t-r)) \rho(t,r)\,dr, & t\in(-\infty,T],\label{genLT}\\
D^{(\rho)}_0 u(t)   &:= \int_0^{t} ( u(t) - u(t-r)) \rho(t,r)\,dr + (u(t) - u(0)) \int_{t}^{\infty}\rho(t,r)\,dr,& t\in(0,T], \label{genC}
\end{align}
and $D^{(\rho)}_0 u(0):=\lim_{t\downarrow 0}D^{(\rho)}_0 u(t)$.
\end{definition}
\begin{remark}
The operator $D_{\infty}^{(\rho)}$ can be seen as the left-sided generalization of the \textit{Marchaud derivative} \cite[eq. (5.57) and (5.58)]{samko}. This operator is also known as the \textit{generator form of fractional derivatives}  \cite{KV0, Meerschaert2012}, or a L\'evy-type generator \cite{Schilling}.
\end{remark}
\begin{example}
 We mention some concrete and popular examples of the nonlocal operators.
\begin{enumerate}[(i)]
\item By setting $\rho(t,r) = -r^{-\alpha-1}/\Gamma(-\alpha)$ with $\alpha\in(0,1)$, the nonlocal operator $D^{(\rho)}_0$ reproduces the Caputo fractional derivative \cite{kai}, and $  D_{\infty}^{(\rho)}$ the Marchaud fractional derivative \cite{samko}.
\item The operator $\mathcal G_\delta$, defined in \cite[formula (1.2)]{DYZ17}, is a special case of the Marchaud-type derivative $D_{\infty}^{(\rho)}$ with a time-independent and compactly supported kernel function.
\item Other particular cases include the  fractional derivatives of variable  order, which are obtained by taking $\rho$ as the function $\rho(t,r) = -r^{-1-\alpha(t)}/\Gamma(-\alpha(t))$
with a suitable function $\alpha(t) : \mathbb{R} \to (0,1)$ \cite{KV-1}, and
tempered L\'evy kernels  $\rho(t,r)=\, -e^{-\lambda r}r^{-1-\alpha}/\Gamma(-\alpha)$, $ \alpha \in (0,1),\, \lambda>0$,  \cite{Meetemp, WyTemp}.
%\item {\color{red}add log-periodic and distributed order??}

\end{enumerate}
\end{example}

\begin{remark}\label{ch_1_rmk_MC_regularity2}
The nonlocal derivatives $-D^{(\rho)}_\infty$ and $-D^{(\rho)}_0$ have a clear probabilistic interpretation. The former tells us that the process at $t$ makes a negative jump of size $|r|$ with intensity $\rho(t,r)$. The latter tells us that, as long as the  jump does not cross $0$, the process jumps from $t$ to $t-r$ with intensity $\rho(t,r)$. Otherwise, it gets killed with rate/intensity $\int_t^\infty\rho(t,r)\,dr$ and regenerated at $0$ with the same rate, where it remains absorbed. This will be made rigorous in Definition \ref{3processes} and Proposition \ref{3processesresults}.
\end{remark}

\subsection{Probabilistic interpretation and preliminary results}
In this section, we discuss three stochastic processes generated by the operators defined in \eqref{genLT} and \eqref{genC}
with kernel functions satisfying \ref{H0}. %{\color{red} rewrite the following definition}
\begin{definition}\label{3processes} Assume \ref{H0}.
\begin{enumerate}[(i)]
	\item \cite[Theorem 5.1.1]{KV0}: Let $T^{(\rho),\infty}=\{T^{(\rho),\infty}_s\}_{s\ge0}$ be the Feller  semigroup on $C_\infty(-\infty,T]$ with the generator
		$$\left(\mathcal L_\infty^{(\rho)}, \text{Dom}(\mathcal L_\infty^{(\rho)})\right)\text{ being the closure of }\left(-D_{\infty}^{(\rho)},C^1_{\infty}(-\infty,T]\right),
		$$
		and recall that $C^1_{\infty}(-\infty,T]$ is invariant under $T^{(\rho),\infty}$. \\
			We denote the induced Feller process by 
		$$
		-X^{t,(\rho)} =\{-X^{t,(\rho)}  (s)\}_{s \ge 0}, \quad t\in (-\infty,T].
		$$

\item \cite[Theorem 4.1]{KVFDE}: Let $T^{(\rho)}=\{T^{(\rho)}_s\}_{s\ge0}$ be the Feller  semigroup  on $C[0,T]$ with the generator 
	$$
	\left(\mathcal L^{(\rho)}, \text{Dom}(\mathcal L^{(\rho)})\right)\text{ being the closure of }\left(-D_{0}^{(\rho)},C^1[0,T]\right),
	$$ 
	and recall that $C^1[0,T]$ is invariant under $T^{(\rho)}$.\\
		We denote the induced Feller process by $-X^{t,(\rho)}_{0}=\{-X^{t,(\rho)} (s)\mathbf 1_{\{s<\tau_0(t)\}}\}_{s \ge 0}$, $t\in[0,T]$.
	
\item We denote by  $T^{(\rho),\text{kill}}=\{T^{(\rho)}_s\}_{s\ge0}$  the Feller  semigroup on $C_0[0,T]$ with the generator
	$$
	\left(\mathcal L_{\text{kill}}^{(\rho)}, \text{Dom}(\mathcal L_{\text{kill}}^{(\rho)})\right)\text{ being the closure of }\left(-D_{0}^{(\rho)},C^1_0[0,T]\right),
	$$ 
and recall that $C^1[0,T]$ is invariant under $T^{(\rho),\text{kill}}$.\\ 
 We denote the induced Feller process by $-X^{t,(\rho),\text{kill}}_{0}=\{-X^{t,(\rho),\text{kill}}_0 (s)\}_{s \ge 0}$, $t\in(0,T]$.

\end{enumerate}
\end{definition}

	\begin{remark}
The next proposition justifies the notation for the stochastic processes $-X^{t,(\rho)}_0$ and  Definition \ref{3processes}-(iii). The proof of parts (i), (ii) and (iii) is given in \cite[Proposition 2.7]{HKT17}, and hence omitted here. Part (iv) can be proved by the same argument for  Lemma \ref{lem32}.
		\end{remark}

\begin{proposition}\label{3processesresults}
 \begin{enumerate}[(i)]
	\item  The processes $-X^{t,(\rho)}$, $-X^{t,(\rho)}_0$ and $-X^{t,(\rho),\text{kill}}_0$ are non-increasing and  
	$$
	\mathbf P[-X^{t,(\rho)}(s)\in(a,b)]=\mathbf P[-X^{t,(\rho)}_0(s)\in(a,b)]=\mathbf P[-X^{t,(\rho),\text{kill}}_0(s)\in(a,b)], $$
	for every $t\in(0,T]$, $0<a<b\le T$, $s> 0$. In particular $\mathbf P[-X^{t,(\rho)}_0 (s)\in \{0\}]=\mathbf P[-X^{t,(\rho)} (s)\le0],$ for every $ t \in [0,T]$, $s> 0$.
	\item The   law of $$\tau_0 (t):=\inf\{s > 0:-X^{t,(\rho)}(s)\le 0\},\quad t\in (-\infty, T],
	$$
	equals the law of the first exit time from the interval $(0,T]$ of the processes $-X^{t,(\rho)}_0$  for each $t\in (0, T]$ (so that we will use indistinctly the same notation  $\tau_0  (t)$).
	\item The expectation of $\tau_0 (t)$ is uniformly bounded, i.e.,  $\sup_{t\in[0,T]}\mathbf E[\tau_0 (t)]<\infty$.
	\item 		 It holds that
$
	\left(\mathcal L_{\text{kill}}^{(\rho)}, \text{Dom}(\mathcal L_{\text{kill}}^{(\rho)})\right)= \left(\mathcal L^{(\rho)}, \text{Dom}(\mathcal L^{(\rho)}\right)\cap\{f(0)=0\}).
	$
		\end{enumerate}
\end{proposition}

\begin{remark}\label{rem:X0}
%We have some observations on processes $-X^{t,(\rho)}_0$ and $-X^{t,(\rho)}$:
\begin{enumerate}[(i)]
\item
	It follows from Proposition \ref{3processesresults} that the process  $-X^{t,(\rho)}_0$ is obtained by absorbing at the point $0$ the process $-X^{t,(\rho)}$  on its first attempt to leave the interval $(0,T]$.
	\item Definition \ref{3processes}-(ii) (Definition \ref{3processes}-(iii)) could be a proposition derived from absorbing  at $0$ (killing on crossing $0$) the process $-X^{t,(\rho)}$, $t>0$.
\item If the L\'evy kernel is independent of $t$, i.e. $\rho(t,r)=\rho(r)$, then $-X^{t,(\rho)}(s)=t-X^{(\rho)}(s)$ is the decreasing L\'evy process with generator $-D_{\infty}^{(\rho)}$ acting on $C_c^\infty(\mathbb R)$, where $X^{(\rho)}$ is the subordinator with L\'evy measure $\rho(r)dr$. This is a consequence of the fact that $\mathcal L_\infty^{(\rho)}=-D_{\infty}^{(\rho)}$ on  $C_c^\infty(\mathbb R)\subset \text{Dom}(\mathcal L_\infty^{(\rho)})$, and \cite[Theorem 2.7]{Schilling}.
\item If the kernel  $\rho(t,r)=\rho(r)$ is integrable, then $-D_{\infty}^{(\rho)}$ is the generator of a decreasing compound Poisson process.
	\end{enumerate}
\end{remark}

\begin{remark}\label{rem:H0-alter}
The assumption \ref{H0} could be replaced with an alternative one, as long as $-D^{(\rho)}_\infty$ generates a non-increasing  Feller process with the first exit times from $(0,T]$ having finite expectation, along with the existence of invariant cores with the properties in Definition \ref{3processes}. Nevertheless, the assumption \ref{H0} provides a satisfactory level of generality for most of the applications we have in mind.
\end{remark}

Finally, we use one more assumption on the stochastic process $X^{t,(\rho)}$.

 % Karol: 1) exit on boundary, 2) invariant core with L=D, 3) move to adjoint, 4) move to adjoint-of-adjoint.
\begin{description}
\item[(H1)\label{H1}] The law of   $-X^{t,(\rho)} (s)$ is absolutely continuous with respect to Lebesgue measure for each $t\in [0,T],$ $ s>0$, and we denote such density by $p_s^{(\rho)}(t)$. Furthermore assume that $\mathbf P[-X^{t,(\rho)}(\tau_0 (t))\in \{0\}]=0$, for each $t\in(0,T]$.
\end{description}

\begin{remark}\label{rem:H12}
Assumption \ref{H1} ensures the existence of the probability density function $p_s^{(\rho)}(t)$, which helps us  handle the weak problem data (see Theorem \ref{lem_sdoggs}-(ii)).
Otherwise,  without  \ref{H1}, we could assume that the problem data $g$ in Theorem \ref{lem_sdoggs}-(ii) is a Baire class 1 function (Remark \ref{rmk:noH1}). This
allows us to handle several cases, such as $\rho$ being integrable \cite[Remark 27.3]{Sato}.
\end{remark}

\begin{remark}
Assumption $\mathbf P[-X^{t,(\rho)}(\tau_0 (t))\in \{0\}]=0$ is implied by the existence of a density $p_s^{(\rho)}(t)$ if $\rho(t,r)dr=\rho(dr)$. This is because  the existence of a density implies that $\rho((0,\infty))=\infty$, as $X^{(\rho)}$ cannot be a compound Poisson process. Then $\tau_0(t)=\inf\{s>0:X^{(\rho)}(s)>t \}$, the right inverse of $X^{(\rho)}$, and one can apply \cite[III, Theorem 4]{bertoin}. Here $X^{(\rho)}$ is the increasing subordinator with L\'evy measure $\rho(dr)$.  %\red{Check L\'evy matters III for $t$-dependence...}
\end{remark}
\begin{example}
We list some  examples where the densities $p_s^{(\rho)}(t),$ $t,s>0$, exist:
%\begin{remark}\label{rem:H1}
%Examples where the densities $\{p_s^{(\rho)}(t)\}_{t,s>0}$ exists:
\begin{enumerate}[(i)]
\item kernels $\rho(t,r)dr=\rho(dr)$ and $\rho(dr)\ge r^{-1-\alpha}dr$ for all small $r$ \cite[Proposition 28.3]{Sato};
\item kernels $\rho(t,r)=\rho(r)$, $\int_0^\infty \rho(r)\,dr=\infty$ \cite[Theorem 27.7]{Sato};
\item kernels $\rho(t,r)$ such that the respective symbols satisfies the H\"older continuity-type conditions  in \cite[Theorem 2.14]{Levymatters6};
\item see \cite{four02} for another set of assumptions for kernels of the type $\rho(t,r)=p(t)q(r)$ and a literature discussion.
\end{enumerate}
%\end{remark}

\end{example}

\subsection{The spatial operator $\Delta$}
\begin{definition}
For a bounded open set  $\Omega\subset \mathbb R^d$ we say that  $z\in\partial \Omega$ \emph{is a regular point for }$\Omega$, if there exists a right circular finite cone with vertex at $z$, denoted by $V_z$,  such that $V_z\subset \Omega^c$.  We say a bounded open set  $\Omega\subset \mathbb R^d$ \emph{is regular} if every  $z\in\partial \Omega$ is a regular point for $\Omega$.
\end{definition}
\begin{remark}
 From now on, we always assume that $\Omega\subset \mathbb R^d$ is a regular set. In particular, every Lipschitz domain is regular.
\end{remark}
\begin{definition}\label{1process} Let $\Omega\subset\mathbb R^d$ be a regular set. Let $(\Delta_\Omega, \text{Dom}(\Delta_\Omega ))$ be the generator of the Feller semigroup  $T^\Omega=\{T^\Omega_s\}_{\ge0}$ on $ C_{\partial\Omega}(\Omega)$, where $T_s^\Omega f(x):=\mathbf E[f(B^x(s))\mathbf 1_{\{s<\tau_\Omega(x)\}}]$, $s\ge 0$, $x\in \overline \Omega$, with $B^x(s)=x+B(2s)$, $s\ge0$, $x\in\Omega$, $\{B(s)\}_{s\ge 0}$ being the standard $d$-dimensional Brownian motion, and define the first exit times 
$$\tau_\Omega(x):=\inf\{s>0:B^x(s)\notin \Omega\},\quad x\in\Omega.$$
\end{definition}

\begin{remark}
Recall that $\text{Dom}(\Delta_\Omega)=\{f\in C_{\partial\Omega}(\Omega)\cap C^2(\Omega):\Delta f\in C_{\partial\Omega}(\Omega)\}$
(see, e.g., \cite[Theorem 2.3]{MeerB}). We write $\Delta_\Omega=\Delta$ from now on.
We denote the law of  $B^x(s)\mathbf 1_{\{s<\tau_\Omega(x)\}}$ by $p^\Omega_s(x,y)dy$, recalling that $(x,y)\mapsto p^\Omega_s(x,y)$ is continuous  for each $s>0$.
\end{remark}

\begin{remark}
The arguments in Section \ref{sec:generaltheory} could be extended to the case where  the Laplacian $\Delta$ is replaced by an
operator whose semigroup on $C_{\partial\Omega}(\Omega)$ allows a density function $p^\Omega_s(x,y)$ with respect to Lebesgue measure for positive time (i.e. the respective version of the first part of assumption \ref{H1}). The restricted fractional Laplacian is an example of such operator (see, e.g., \cite{Bogdan,BonVas16}).
\end{remark}

\subsection{The inhomogeneous Caputo-type evolution equation}
In order to study the stochastic representation of solution of problem (\ref{preRL}), we consider the following equivalent form
\begin{equation}\label{postRL}
\left\{
\begin{split}
(-D^{(\rho)}_0+\Delta) u(t,x)&=-g(t,x), & \text{in }&(0,T]\times  \Omega,\\
 u(t,x)&=0, & \text{in }&(0,T]\times \partial\Omega, \\
 u(0,x)&=\phi(0,x), & \text{in }&  \Omega,
\end{split}
\right.
\end{equation}
with  the forcing term $g=f+f_\phi$, where we define
\[
 f_\phi(t,x):=\int_{t}^\infty (\phi(t-r,x)-\phi(t,x))\rho(t,r)\,dr, \quad \text{in }(0,T]\times  \Omega.
\]
%for some $\phi\in B((-\infty,0]\times\Omega)$.
Notice that  $f_\phi=-D_{\infty}^{(\rho)}\phi$, for $\phi$ extended to $\phi(0)$ on $(0,T]\times
\overline\Omega$, and $D^{(\rho)}_\infty u = D^{(\rho)}_0u-f_\phi$ for any smooth $u$ such that $u=\phi$ on $(-\infty,0]$.
%Problem (\ref{postRL}) has been extensively studied, we refer to \cite{Chen17, CK18,  DYZ17, HKT17}.
In the following section, we shall discuss the probabilistic representation of the solution to \eqref{preRL} with the help of
the reformulation \eqref{postRL}, provided certain hypothesis on problem data. Let us mention that versions of the Caputo-type problem \eqref{postRL} have also been studied in \cite{Chen17, CK18, DYZ17, HKT17}.

\section{General theory}\label{sec:generaltheory}

In this part, we study the solution theory of the nonlocal problem. To this end, we begin with the study on some  time-space compound semigroups which are constructed using temporal semigroups and spatial ones. This allows us to treat the Caputo-type EE \eqref{postRL} as an elliptic boundary value problem.

\subsection{Time-space compound semigroups}
The next lemma shows that $\{T^{(\rho)}_sT^\Omega_s\}_{s\ge 0}$ is a well-defined Feller semigroup on $C_{\partial\Omega}([0,T]\times\Omega)$ such that its generator is the closure of $-D^{(\rho)}_0+\Delta$.
\begin{lemma}\label{lem:31}
With the notation of Definition \ref{3processes} and Definition \ref{1process}, the operators
$$T^{(\rho),\Omega}:=\{T^{(\rho)}_sT^\Omega_s\}_{s\ge 0}$$
form a Feller semigroup on $C_{\partial\Omega}([0,T]\times\Omega)$, whose generator $(\mathcal L_\Omega^{(\rho)},\text{Dom}(\mathcal L^{(\rho)}_\Omega))$ is the closure of
$$
\left(-D^{(\rho)}_0+\Delta,\ \text{Span}\left\{C^1[0,T]\cdot \text{Dom}(\Delta_\Omega)\right\} \right)\quad\text{in}\quad C_{\partial\Omega}([0,T]\times\Omega),
$$
where $T^{(\rho)}$ and $-D^{(\rho)}_0$ act on the $[0,T]$-variable, and $T^\Omega$ and $\Delta$ act on the $\Omega$-variable.
\end{lemma}
\begin{proof}
It is straightforward to show that $T^{(\rho),\Omega}$ is a Feller semigroup by observing
that
$$T^{(\rho)}_sT^\Omega_r=T^\Omega_rT^{(\rho)}_s,\qquad \text{for every} ~~ s,r\ge 0,$$
and the contraction property
\[
\|T^\Omega_s f\|_{C([0,T]\times\overline\Omega)}\le \|f\|_{C([0,T]\times\overline\Omega)} \quad \text{and} \quad \|T^{(\rho)}_s f\|_{C([0,T]\times\overline\Omega)}\le \|f\|_{C([0,T]\times\overline\Omega)},
\]
holds for every $f\in C_{0,\partial\Omega}([0,T]\times\Omega)$, $s\ge0$. We denote the generator of $T^{(\rho),\Omega}$ by $(\mathcal L_{\Omega}^{(\rho)},\text{Dom}(\mathcal L_{\Omega}^{(\rho)}))$. Let $f=pq $, where $p\in C^1[0,T]$ and $q\in\text{Dom}(\Delta_\Omega) $. Then $\mathcal L^{(\rho)}p=-D^{(\rho)}_0p$ from Definition \ref{3processes}-(ii), and by a standard  triangle inequality argument, we obtain
\begin{align*}
\Big|\frac{T^{(\rho)}_hT^\Omega_h f(t,x)-f(t,x)}{h}&-(-D^{(\rho)}_0 +\Delta )f(t,x)\Big|\\
\le&\ \|p\|_{C[0,T]}\left\|\frac{T^\Omega_hq-q}{h} -\Delta q\right\|_{C(\overline\Omega)} + \|\Delta q\|_{C(\overline\Omega)}\left\| T^{(\rho)}_h p-p\right\|_{C[0,T]} \\
&\ +\|q\|_{C(\overline\Omega)}\left\| \frac{T^{(\rho)}_h p-p}{h}+D^{(\rho)}_0p\right\|_{C[0,T]}\to 0,
\end{align*}
as $h \downarrow 0$. As a result, $\mathcal L_{\Omega}^{(\rho)}= (-D^{(\rho)}_0 +\Delta)$ on $\text{Span} \{C^1[0,T]\cdot \text{Dom}(\Delta_\Omega) \}\subset  \text{Dom}(\mathcal L_{\Omega}^{(\rho)})$.

Next, we aim to show that $\text{Span} \{C^1[0,T]\cdot \text{Dom}(\Delta_\Omega)\}$ is dense in $ C_{\partial\Omega}([0,T]\times\Omega)$. It is enough to show that $\text{Span} \{C^\infty[0,T]\cdot C^\infty_c(\Omega)\}$ is dense  in $ C_{\partial\Omega}([0,T]\times\Omega)$ by the inclusion
$$\text{Span} \{C^\infty[0,T]\cdot C^\infty_c(\Omega)\}\subset \text{Span} \{C^1[0,T]\cdot \text{Dom}(\Delta_\Omega) \}.$$

To this end, we notice that $\text{Span} \{C^\infty[0,T]\cdot C^\infty(\overline\Omega)\}$ is a sub-algebra of $C([0,T]\times\overline\Omega)$
that contains constant functions and separates points. Hence  $\text{Span} \{C^\infty[0,T]\cdot C^\infty(\overline\Omega)\}$ is dense in $C([0,T]\times\overline\Omega)$
by Stone-Weierstrass Theorem for compact  Hausdorff spaces.
Then for   $f\in  C_{\partial\Omega}([0,T]\times\Omega)$ we take a sequence
$\{f_n\}_{n\in\mathbb N}\subset \text{Span} \{C^\infty[0,T]\cdot C^\infty(\overline\Omega)\} $
such that $f_n\to f$. Pick functions $\{ 1_{K_n}\}_{n\in\mathbb N}\subset C_c^\infty(\Omega)$
such  that $ 0\le 1_{K_n}\le 1 $, $  1_{K_n}(x)=1$ for  $x\in K_n$, and $  1_{K_n}(x)=0$ for  $x\in\Omega\backslash K_{n+1}$,
where $K_n$ is compact and $K_n\subset K_{n+1}\subset \Omega$ for each $n$, and $\cup_n K_n=\Omega$.
Define $\tilde f_n:= 1_{K_n}f_n\in  \text{Span} \{C^\infty[0 ,T]\cdot C_c^\infty(\Omega)\}$ for each  $n\in\mathbb N$.   Then, as $n\to\infty$
\begin{align*}
\|\tilde f_n -f\|_{C ([0,T]\times \Omega)}
&\le\ \|\tilde f_n -f\|_{C ([0,T]\times K_n)} + \|\tilde f_n -f\|_{C \left( [0,T]\times K_{n+1}\backslash K_n\right)}+\|\tilde f_n - f\|_{C \left([0,T]\times \overline\Omega \backslash K_{n+1}\right)}\\
&=  \|  f_n -f\|_{C ([0,T]\times K_n)} + \|\tilde f_n -f\|_{C \left( [0,T]\times K_{n+1}\backslash K_n\right)}+\|f\|_{C \left([0,T]\times \overline\Omega \backslash K_{n+1}\right)}\\
&\to0.
\end{align*}
Then the density of $\text{Span} \{C^1[0,T]\cdot \text{Dom}(\Delta_\Omega)\}$ in $ C_{\partial\Omega}([0,T]\times\Omega)$
together with the fact that $\text{Span} \{C^1[0,T]\cdot \text{Dom}(\Delta_\Omega) \}$ is invariant under $T^{(\rho),\Omega}$ and a subspace of $\text{Dom}(\mathcal L_\Omega^{(\rho)})$
 completes the proof by \cite[Lemma 1.34]{Schilling}.
\end{proof}

Then a similar argument shows the following corollary.
\begin{corollary}\label{cor32}
With the notation of Definition \ref{3processes} and Definition \ref{1process}, it holds that:
\begin{enumerate}[(i)]
\item
 the operators $T^{(\rho),\text{kill},\Omega}:=\{T^{(\rho),\text{kill}}_sT^\Omega_s\}_{s\ge 0}$ form a Feller semigroup on $C_{0,\partial\Omega}([0,T]\times\Omega)$. The generator $(\mathcal L_{\Omega}^{(\rho),\text{kill}},\text{Dom}(\mathcal L_{\Omega}^{(\rho),\text{kill}}))$ of $T^{(\rho),\text{kill},\Omega}$ is the closure of
$$
\left(-D^{(\rho)}_0+\Delta,\ \text{Span}\{C^1_0[0,T]\cdot \text{Dom}(\Delta_\Omega)\}\right)\quad\text{in}\quad C_{0,\partial\Omega}([0,T]\times\Omega),
$$
 where $T^{(\rho),\text{kill}}$ and $-D^{(\rho)}_0$ act on the $[0,T]$-variable, and $T^\Omega$ and $\Delta$ act on the $\Omega$-variable.
\item \label{cor:32}
The operators $T^{(\rho),\infty,\Omega}:=\{T^{(\rho),\infty}_sT^\Omega_s\}_{s\ge 0}$ form a Feller semigroup on $C_{\infty,\partial\Omega}((-\infty,T]\times\Omega)$. The generator $(\mathcal L_{\Omega}^{(\rho),\infty},\text{Dom}(\mathcal L_{\Omega}^{(\rho),\infty}))$ of $T^{(\rho),\infty,\Omega}$ is the closure of
$$
\left(-D_{\infty}^{(\rho)}+\Delta,\ \text{Span}\{C^1_\infty(-\infty,T]\cdot \text{Dom}(\Delta_\Omega)\} \right)\quad\text{in}\quad C_{\infty,\partial\Omega}((-\infty,T]\times\Omega),
$$
where $T^{(\rho),\infty}$ and $-D_{\infty}^{(\rho)}$ act on the $(-\infty,T]$-variable, and $T^\Omega$ and $\Delta$ act on the $\Omega$-variable.
\end{enumerate}
\end{corollary}

\begin{remark}
If the spatial generator is not the Laplacian, it could happen that $C_c^\infty(\Omega)$ is not contained in the domain of the spatial generator (as in the case of the restricted fractional Laplacian). In such case one can extend the proof of Lemma \ref{lem:31} as in \cite[Appendix II]{T18}.
\end{remark}

\begin{lemma}\label{lem32}
With the notation of Definition \ref{3processes} and Definition \ref{1process}, it holds that
$$T^{(\rho),\Omega} =T^{(\rho),\text{kill},\Omega}\quad \text{on}\quad C_{0,\partial\Omega}([0,T]\times \Omega),$$
and
$$\mathcal L_{\Omega}^{(\rho)} =\mathcal L_{\Omega}^{(\rho),\text{kill}}\quad\text{on}\quad\text{Dom}(\mathcal L_{\Omega}^{(\rho),\text{kill}})
 = \text{Dom}(\mathcal L_{\Omega}^{(\rho)})\cap\{f(0)=0\}.$$
\end{lemma}
\begin{proof}
 The first  claim  is an immediate consequence of the observation that
 $T^{(\rho),\text{kill}}=T^{(\rho)}$ on $C_{0}[0,T].$
To prove the second claim, we first confirm  that
$\text{Dom}(\mathcal L_{\Omega}^{(\rho),\text{kill}})\subset \text{Dom}(\mathcal L^{(\rho)}_\Omega )$
by the fact that $T_s^{(\rho),\Omega}=T^{(\rho),\text{kill},\Omega}$ on $C_{0,\partial\Omega}([0,T]\times \Omega)$.
Next, we show that
\begin{equation*}
u-u(0)\in \text{Dom}(\mathcal L_{\Omega}^{(\rho),\text{kill}})\quad \text{for all} \quad u\in \text{Dom}(\mathcal L^{(\rho)}_\Omega ).
\end{equation*}
In fact, let $u\in \text{Dom}(\mathcal L^{(\rho)}_\Omega )$ and consider its resolvent representation for some $\lambda>0$ and $g\in C_{\partial\Omega}([0,T]\times \Omega)$
\[
u(t,x)=\int_0^\infty e^{-\lambda s}T^{(\rho)}_sT^\Omega_sg(t,x)\, ds,
\]
and hence
\begin{align*}
u(t,x)-u(0,x)&=\int_0^\infty e^{-\lambda s}T^\Omega_sT^{(\rho)}_s(g-g(0))(t,x)\,ds\\
&=\int_0^\infty e^{-\lambda s}T^\Omega_sT^{(\rho),\text{kill}}_s(g-g(0))(t,x)\,ds\in \text{Dom}(\mathcal L_{\Omega}^{(\rho),\text{kill}}),
\end{align*}
where we use the fact that  $T^{(\rho),\text{kill}}=T^{(\rho)}$ on $C_{0,\partial\Omega}([0,T]\times \Omega)$ and that $g-g(0)\in C_{0,\partial\Omega}([0,T]\times \Omega)$.

\end{proof}

\begin{remark}\label{rem:resol}
Note that the resolvent representation yields that
\[
\left(-\mathcal L^{(\rho),\text{kill}}_\Omega\right)^{-1}g(t,x)=\int_0^\infty T_s^{(\rho),\Omega}g(t,x)\,ds=\mathbf E\left[\int_0^{\tau_0(t) \wedge \tau_\Omega(x)}g\left(-X^{t,(\rho)}(s),B^x(s)\right)ds\right],
\]
for $g\in C_{0,\partial\Omega}([0,T]\times \Omega)$, as
\begin{align*}
T_s^{(\rho),\Omega}g(t,x)= T_s^{(\rho)}T^{\Omega}_sg(t,x)&=\mathbf E\left[g\left(-X^{t,(\rho)}(s)\mathbf 1_{\{s<\tau_0(x)\}},B^x(s)\mathbf 1_{\{s<\tau_\Omega(x)\}}\right)\right]\\
&= \mathbf E\left[g\left(-X^{t,(\rho)}(s),B^x(s)\right)\mathbf 1_{\{s<\tau_0(x)\}}\mathbf 1_{\{s<\tau_\Omega(x)\}}\right].
\end{align*}
 Also, if $g=1$ then $(-\mathcal L^{(\rho),\text{kill}}_\Omega)^{-1}g(t,x)=\mathbf E[\tau_{t,x}]$, where we write $\tau_{t,x}=\tau_0 (t) \wedge \tau_\Omega(x)$.
\end{remark}

\subsection{Notions of solutions}
In order to discuss the stochastic representation of solutions to \eqref{preRL},
we use the following two auxiliary notions of solutions to the variant problem \eqref{postRL}, as in \cite{HKT17}.
\begin{definition}\label{def:sdog}
Let $g\in C_{\partial\Omega}([0,T]\times \Omega)$ and $\phi(0)\in \text{Dom}(\Delta_\Omega)$ such that $g(0)=-\Delta \phi(0)$. We say that a function $u\in C_{\partial\Omega}([0,T]\times \Omega)$ is a \emph{solution in the domain of the generator to problem }(\ref{postRL}) if
\begin{equation}
\mathcal L_{\Omega}^{(\rho)}u=-g\text{ on }(0,T]\times\overline\Omega,\quad u(0)=\phi(0),\quad\text{and } u\in \text{Dom}(\mathcal L^{(\rho)}_\Omega).
\label{sdgcond}
\end{equation}
\end{definition}
%\begin{remark}(uniqueness of cts decaying +maximum principle?)\end{remark}
The next solution concept for problem (\ref{postRL}) is defined as a pointwise approximation of solutions  in the domain of the generator.
%such that the approximating initial data $\{g_n\}_{n\in\mathbb N}$ satisfies a dominated convergence type of condition (in order to guarantee the existence of an appropriate sequence {\color{red}$\{g_n\}_{n\in\mathbb N}$ we will assume (H1)}).
\begin{definition}\label{def:gen_sol}
Let $g\in B([0,T]\times \Omega)$ and $\phi(0)\in \text{Dom}(\Delta_\Omega)$. We say that a function $u\in B([0,T]\times \overline\Omega)$ is a \emph{generalized solution to problem }(\ref{postRL}) if
\[
u=\lim_{n\to\infty}u_n\quad\text{pointwise,}
\]
where  $\{u_n\}_{n\in\mathbb N}$ is a sequence of solutions in the domain of the generator for a corresponding sequence of data $\{g_n\}_{n\in\mathbb N}\subset C_{\partial\Omega}([0,T]\times \Omega)$ such that $g_n\to g$ a.e. on $(0,T]\times \Omega$,  $\sup_n \|g_n\|_\infty<\infty$, and $g_n(0)=-\Delta\phi(0)$ for each $n\in\mathbb N$.
\end{definition}

\begin{remark}\label{rem:genb}
The generalized solution will retain the homogeneous Dirichlet boundary condition on $\partial\Omega$ and the initial condition  $u(0)=\phi(0)$.
\end{remark}

\subsection{Well-posedness and Feynman-Kac  formula for  problem (\ref{postRL})}
In order to study the Feynman-Kac stochastic formula, we use following assumption on the initial data:
\begin{description}
\item[(H2)\label{H2}] The initial data  $\phi:(-\infty,0]\times \overline\Omega\to \mathbb R$ is such that the extension of $\phi$ to $\phi(0)$ on $(0,T]\times \overline \Omega$ satisfies
$\phi\in \text{Dom}(\mathcal L^{(\rho),\infty}_\Omega)$ and $ \mathcal L^{(\rho),\infty}_\Omega \phi = (-D_{\infty}^{(\rho)}+\Delta)\phi$.
%	with the notation of Lemma \ref{thm_main_SR}-(iii).
\end{description}

\begin{remark}\label{rem:H2}
We have some observations on the assumption \ref{H2}:
\begin{enumerate}[(i)]
	\item Assumption \ref{H2} is satisfied for example by  linear combinations of  initial conditions
	in variables-separable form, that is,  $\phi(t,x)=p(t)q(x)$,
	where $p\in C^1_\infty((-\infty,0])$, $p'(0-)=0$ and $q\in \text{Dom}(\Delta_\Omega)$. Such set of functions
	is dense in $C_{\infty,\partial\Omega}((-\infty,0]\times\Omega)$.
	The problem \eqref{preRL} with such a kind of initial data has been analytically studied in \cite{DYZ17}.
	\item
Note that \ref{H2} implies $\phi(0)\in \text{Dom}(\Delta_\Omega)$ and $f_\phi \in C([0,T]\times\Omega)$.
This is because \ref{H2} implies $\phi(0)\in C_{\partial\Omega}(\Omega)$, $\Delta\phi(t)=\Delta\phi(0)\in C_{\partial\Omega}(\Omega)$ for $t\in[0,T]$  and $f_\phi =-D_{\infty}^{(\rho)}\phi$.
\item The case where \ref{H2} no longer holds is to be discussed in the next section.
%	\item We will drop (H2) in Theorem \ref{thm:main}.
	\end{enumerate}
\end{remark}

\begin{theorem}\label{lem_sdoggs} Assume \ref{H0}.  Then
\begin{enumerate}[(i)]
	\item If $g+\Delta \phi(0)\in C_{0,\partial\Omega}([0,T]\times \Omega)$ for some $g\in C_{\partial\Omega}([0,T]\times \Omega)$ and $ \phi(0)\in \text{Dom}(\Delta_\Omega)$,
	then there exists a unique  solution in the domain of the generator to problem (\ref{postRL}).\\
	\item Assume \ref{H1}. If $g\in B([0,T]\times \Omega)$ and $ \phi(0)\in \text{Dom}(\Delta_\Omega)$, then there exists a
	unique generalized  solution  to problem (\ref{postRL}), and the generalized solution allows the stochastic representation for any $(t,x)\in(0,T]\times\Omega$
		\begin{equation}\label{thm:Caputo_SR_C}\begin{split}
 u(t,x)=&\  \mathbf E\left[ \phi(0,B^x(\tau_0(t)))\mathbf 1_{\{\tau_0(t)< \tau_\Omega(x)\}}\right] \\
 &\quad+\mathbf E\left[\int_0^{\tau_0(t)\wedge \tau_\Omega(x)}g\left(-X^{t,(\rho)}(s),B^x(s)\right)ds\right].
\end{split}\end{equation}
	\item Assume \ref{H1}, \ref{H2} and let $g=f+f_\phi$, for $f\in B([0,T]\times\overline\Omega)$. Then both solutions in part (i) and (ii) allow the stochastic representation for any $(t,x)\in(0,T]\times\Omega$
{\small
	\begin{equation}\label{thm:Caputo_SR}\begin{split}
 u(t,x)=&\ \mathbf E\left[ \phi\left(-X^{t,(\rho)}(\tau_0(t)),B^x(\tau_0(t))\right)\mathbf 1_{\{\tau_0(t)< \tau_\Omega(x)\}}\right]\\
  &\quad +\mathbf E\left[\int_0^{\tau_0(t)\wedge \tau_\Omega(x)}f\left(-X^{t,(\rho)}(s),B^x(s)\right)ds\right].
\end{split}\end{equation}}
\end{enumerate}
\end{theorem}\vskip10pt

\begin{proof}
(i) Recall that we write $\tau_{t,x}=\tau_0(t)\wedge\tau_\Omega(x)$. Then using Proposition \ref{3processesresults}-(iii)  with the inequality
\begin{equation*}
\left|(-\mathcal L_{\Omega}^{(\rho),\text{kill}})^{-1}w(t,x)\right|=\left|\mathbf E\left[\int_0^{\tau_{t,x}}w\left(-X^{t,(\rho)}(s),B^x(s)\right)ds)\right]\right|\le \|w\|_\infty\mathbf E\left[\tau_{t,x}\right],
\end{equation*}
for any bounded $w$, we know that $(-\mathcal L_{\Omega}^{(\rho),\text{kill}})^{-1}$
is bounded on $C_{0,\partial\Omega}([0,T]\times \Omega)$.
Meanwhile, we observe that $T^{(\rho),\text{kill},\Omega}_sw\in C_{0,\partial\Omega}([0,T]\times \Omega)$ if $w\in C_{0,\partial\Omega}([0,T]\times \Omega)$ for each $s>0$, and it holds that
$$
\int_0^\infty \left|T^{(\rho),\text{kill},\Omega}_sw(t,x)\right|ds\le \int_0^\infty \|w\|_\infty\mathbf P[s< \tau_{t,x}]ds=\|w\|_\infty\mathbf E[\tau_{t,x}]<\infty.
$$
Therefore we conclude that $(-\mathcal L_{\Omega}^{(\rho),\text{kill}})^{-1}$ maps $C_{0,\partial\Omega}([0,T]\times \Omega)$ to itself.
Then it follows by \cite[Theorem 1.1']{Dyn65} that $\bar u:=(-\mathcal L_{\Omega}^{(\rho),\text{kill}})^{-1}(g+\Delta \phi(0))$ is the unique solution to
\begin{equation}
\mathcal L_{\Omega}^{(\rho),\text{kill}} \bar u = -(g+\Delta \phi(0)) \text{ on }(0,T]\times\overline\Omega,\quad \bar u(0)=0,\quad\text{and  $\bar u\in \text{Dom}(\mathcal L_{\Omega}^{(\rho),\text{kill}})$}.
\label{aaa}
\end{equation}
It remains to show that $u$ satisfies (\ref{sdgcond}) if and only if $u-\phi(0)$ satisfies (\ref{aaa}). For the `if' direction, let $\bar u$ satisfy \eqref{aaa}. Then $u:=\bar u+\phi(0)\in \text{Dom}(\mathcal L_{\Omega}^{(\rho)})$ and $\mathcal L_{\Omega}^{(\rho),\text{kill}}\bar u=\mathcal L_{\Omega}^{(\rho)}\bar u$, both by Lemma \ref{lem32}. Also $\mathcal L_{\Omega}^{(\rho)}\phi(0)=\Delta\phi(0)$ by Lemma \ref{lem:31}, using $\mathcal L^{(\rho)}1=0$. To conclude observe that by \eqref{aaa}, $u(0)=\phi(0)$ and
$$
\mathcal L_{\Omega}^{(\rho)}(\bar u+\phi(0))=\mathcal L_{\Omega}^{(\rho),\text{kill}}\bar u+\Delta\phi(0) = -g.
$$
The `only if' direction is similar and omitted.

(ii) Now we let $g\in B([0,T]\times \Omega)$ and $ \phi(0)\in \text{Dom}(\Delta_\Omega)$.
Then we can take  a sequence $\{g_n\}_{n\mathbb N}\in C_{0,\partial\Omega}([0,T]\times \Omega) $
such that $g_n\to g$ a.e., $\sup_n\| g_n\|_\infty<\infty$ and  $g_n(0)=-\Delta \phi(0)$ as required by Definition \ref{def:gen_sol}.
Now for each $g_n$, by Remark \ref{rem:resol}, we consider the stochastic representation of the respective solution in the domain  of the generator
\begin{equation*}
u_n(t,x) =\mathbf E\left[\int_0^{\tau_{t,x}} g_n\left(-X^{t,(\rho)}(s),B^x(s)\right) ds\right]+\mathbf E\left[\int_0^{\tau_{t,x}}\Delta \phi(0,B^x(s)) ds\right]+\phi(0,x).
\end{equation*}
Then for any $(t,x)\in (0,T]\times\Omega$, we note that
\begin{align*}
\mathbf E\left[\int_0^{\tau_{t,x}} g_n\left(-X^{t,(\rho)}(s),B^x(s)\right) ds\right]&=\int_0^\infty\mathbf E\left[g_n\left(-X^{t,(\rho)}(s),B^x(s\wedge \tau_\Omega(x))\right)\mathbf 1_{\{s<\tau_0 (t)\}}\right]\, ds\\
&= \int_0^\infty\left( \int_\Omega\int_{(0,t]} g_n(z,y)  p^{(\rho)}_s(t,z)p^\Omega_s(x,y)\,dz\,dy\right) ds\\
& \le  \sup_n\| g_n\|_\infty \mathbf E\left[\tau_{t,x}\right] <\infty,
\end{align*}
where we use the first part of \ref{H1}  and the density $p^\Omega_s$ in the last equality.
Hence we can apply the Dominated Convergence Theorem  to obtain as $n\to\infty$,
\begin{align*}
\mathbf E\left[\int_0^{\tau_{t,x}} g_n\left(-X^{t,(\rho)}(s),B^x(s)\right) ds\right]\rightarrow\mathbf E\left[\int_0^{\tau_{t,x}}g\left(-X^{t,(\rho)}(s),B^x(s)\right) ds\right].
\end{align*}
It follows that a generalized solution $u$ exists and it is given by
\begin{align*}
  u (t,x)&= \mathbf E\left[\int_0^{\tau_{t,x}}g\left(-X^{t,(\rho)}(s),B^x(s)\right) ds\right] + \mathbf E\left[ \int_0^{\tau_{t,x}}\Delta \phi(0,B^x(s)) ds\right]+ \phi(0,x)\\
&= \mathbf E\left[\int_0^{\tau_{t,x}}g\left(-X^{t,(\rho)}(s),B^x(s)\right) ds\right]+ \mathbf E\left[ \phi(0,B^x(\tau_{t,x}))\right],
\end{align*}
where Dyinkin formula (see \cite[Theorem 5.1]{Dyn65}) is used in the last equality. Finally, the uniqueness of the generalized solution follows immediately
from the independence of the approximating sequence.\\

%As a result, the generalized solution satisfies the following representation
%\begin{equation*}
%u(t,x) =\mathbf E\left[\int_0^{\tau_{t,x}} g(-X^{t,(\rho)}(s),B^x(s)) ds\right]+\mathbf E\left[\int_0^{\tau_{t,x}}\Delta \phi(0,B^x(s)) ds\right]+\phi(0,x).
%\end{equation*}

(iii)  Extend $\phi$ to $\phi(0)$ on $(0,T]\times \overline \Omega$, and denote it again by $\phi$.
Then by  Dynkin formula (\cite[Theorem 5.1]{Dyn65}) and Corollary \ref{cor32}-(ii) provided assumption \ref{H2},
we have
\begin{align*}
\mathbf E\left[ \phi\left(-X^{t,(\rho)}(\tau_{t,x}),B^x(\tau_{t,x})\right)\right]- \phi(t,x)&= \mathbf E \left[\int_0^{\tau_{t,x}} (-D_{\infty}^{(\rho)}+ \Delta ) \phi\left(-X^{t,(\rho)}(s),B^x(s)\right) ds\right].
\end{align*}
Meanwhile, for $(t,x)\in(0,T]\times\overline \Omega$  the identities
 $f_\phi(t,x)= -D_{\infty}^{(\rho)}\phi(t,x)$,  $\Delta\phi (0,x)=\Delta\phi (t,x)$ and
\begin{align*}
\int_0^t  (\phi(t-r,x)-\phi(t,x))\rho(t,r)\,dr= \int_0^t  (\phi(0,x)-\phi(0,x))\rho(t,r)\,dr=0
\end{align*}
hold, and we can derive the equality
\begin{align*}
\mathbf E \left[\int_0^{\tau_{t,x}} (-D_{\infty}^{(\rho)}+ \Delta ) \phi\left(-X^{t,(\rho)}(s),B^x(s)\right) ds\right]
&= \mathbf E\left[ \int_0^{\tau_{t,x}} (f_\phi+\Delta \phi)\left(-X^{t,(\rho)}(s),B^x(s)\right) ds\right].
\end{align*}
Therefore, the generalized solution allows the following representation
{\small
\begin{align*}
  u(t,x) &= \mathbf E\left[\int_0^{\tau_{t,x}}\Delta\phi\left(-X^{t,(\rho)}(s),B^x(s)\right)ds\right] +\phi(0,x)+\mathbf E\left[\int_0^{\tau_{t,x}}(f_\phi +f)\left(-X^{t,(\rho)}(s),B^x(s)\right)ds\right]\\
&=\mathbf E \left[\int_0^{\tau_{t,x}} (-D_{\infty}^{(\rho)}+ \Delta ) \phi\left(-X^{t,(\rho)}(s),B^x(s)\right) ds\right]  + \phi(0,x) + \mathbf E\left[\int_0^{\tau_{t,x}}f\left(-X^{t,(\rho)}(s),B^x(s)\right)ds\right]\\
&=\mathbf E\left[ \phi\left(-X^{t,(\rho)}(\tau_{t,x}),B^x(\tau_{t,x})\right)\right]+ \mathbf E\left[\int_0^{\tau_{t,x}}f\left(-X^{t,(\rho)}(s),B^x(s)\right)ds\right]+ \phi(0,x)- \phi(t,x)\\
&= \mathbf E\left[ \phi\left(-X^{t,(\rho)}(\tau_0(t)),B^x(\tau_0(t))\right)\mathbf 1_{\{\tau_0(t)< \tau_\Omega(x)\}}\right] + \mathbf E\left[\int_0^{\tau_{t,x}}f\left(-X^{t,(\rho)}(s),B^x(s)\right)ds\right].
\end{align*}}
for all $(t,x)\in(0,T]\times\Omega$. This completes the proof of the theorem.
\end{proof}
\begin{remark}\label{rmk:noH1}
If assumption \ref{H1} does not hold, one shall modify the definition of a generalized solution requiring pointwise convergence everywhere on $(0,T]\times\Omega$ of the approximating sequence. This allows to run the argument of Theorem \ref{lem_sdoggs}-(ii) as long as one such sequence exists. This means that our data $g$ has to be a Baire class $1$  function
(which includes continuous functions but it is a smaller class than $B([0,T]\times\overline\Omega)$).
\end{remark}

\begin{remark}\label{rmk:lpconv1}
Note that every generalized solution is the pointwise limit on $[0,T]\times\overline\Omega$ of a sequence of solutions in the domain of the generator $\{u_n\}_{n\in\mathbb N}$, and from the stochastic representation we can infer that $\sup_n\|u_n\|_{C([0,T]\times\overline\Omega)}<\infty$. This implies the convergence $u_n\to u$ in $L^p((0,T)\times\Omega)$ for every $p\in[1,\infty)$.
\end{remark}
We now give a more explicit formula for the heat kernel of the solution in \eqref{thm:Caputo_SR} ($f=0$).
\begin{proposition}\label{prop:HK}
Let assumptions \ref{H0} and \ref{H1} hold true. Then
\begin{equation}
  \mathbf E\left[ \phi\left(-X^{t,(\rho)}(\tau_0(t)),B^x(\tau_0(t))\right)\mathbf 1_{\{\tau_0(t)<
\tau_\Omega(x)\}}\right]=\int_{-\infty}^0\int_\Omega \phi(r,y)H_{t,x}(r,y)\,dr\,dy,
\label{SRin}
\end{equation}
for every $(t,x)\in(0,T]\times\Omega$ and $\phi\in B((-\infty,0]\times\Omega)$, where
\[
H_{t,x}(r,y)=\int_0^t \rho(z,z-r)\left(\int_0^\infty p^\Omega_s(x,y)p_s^{(\rho)}(t,z)\,ds\right)dz.
\]
%
%\color{red}
%\begin{align}
% 1&\stackrel{\phi\le1}{\ge}   \mathbf E\left[ \phi\left(-X^{t,(\rho)}(\tau_0(t)),B^x(\tau_0(t))\right)\mathbf 1_{\{\tau_0(t)<
%\tau_\Omega(x)\}}\right]\\
%& \stackrel{\phi=1}{=} \mathbf E_{(-X^{t,(\rho)}(\tau_0(t)),B^x(\tau_0(t))}\left[ 1\right]\\
%&=\int_{-\infty}^0\int_\Omega 1 H_{t,x}(r,y)\,dr\,dy\\
%&=\int_{-\infty}^0\int_\Omega 1\left(\int_0^t \rho(z,z-r)\left(\int_0^\infty p^\Omega_s(x,y)p_s^{(\rho)}(t,z)\,ds\right)dz\right)\,dr\,dy\\
%&=\int_{-\infty}^0 1\left(\int_0^t \rho(z,z-r)\left(\int_0^\infty\left(\int_\Omega p^\Omega_s(x,y)\right)p_s^{(\rho)}(t,z)\,ds\right)dz\right)\,dr\,dy\\
%&=\int_{-\infty}^0 1\left(\int_0^t \rho(z,z-r)\left(\int_0^\infty\left(\mathbf E_{B_s^{x,kill}[1]}\right)p_s^{(\rho)}(t,z)\,dsdz\right)\,dr\,dy\\
%& =C \int_{-\infty}^0 1\left(\int_0^t \rho(z,z-r)\int_0^\infty  p_s^{(\rho)}(t,z)\,dsdz\right)\,dr\,dy\\
%& =C\int_{-\infty}^0 1\left(\int_0^t \rho(z,z-r)\int_0^\infty p_s^{(\rho)}(t,z)\,dsdz\right)\,dr\,dy
%\end{align}
%$C:=\mathbf E_{B_s^{x,kill}[1]}\le 1$
%\color{black}
\end{proposition}
\begin{proof}
By \ref{H1}, it is enough to prove formula (\ref{SRin})
on the set $\{-X^{t,(\rho)}(\tau_0(t))<0\}$.
Fix $(t,x)\in (0,T]\times\Omega$. Let $\phi\in \text{Span}\{C^1_\infty(-\infty,T]\cdot \text{Dom}(\Delta_\Omega)\}$
such that  $\phi=0$ on $[-n^{-1},T]$ for $n\in\mathbb N$. By Remark \ref{rem:H2}-(i) $\phi$ satisfies \ref{H2}. Then by
Dynkin formula along with $\mathcal L_\Omega^{(\rho),\infty}\phi=(-D^{(\rho)}_{\infty}+\Delta)\phi $ by Corollary \ref{cor32} and $\Delta \phi=0$ on $(0,T]$,
we have that
\begin{align*}
  u(t,x):&= \mathbf E\left[ \phi\left(-X^{t,(\rho)}(\tau_0(t)),B^x(\tau_0(t))\right)\mathbf 1_{\{\tau_0(t)< \tau_\Omega(x)\}}\right]\\
&= \mathbf E\left[\int_0^{\tau_{t,x}} -D^{(\rho)}_\infty\phi\left(-X^{t,(\rho)}(s),B^x(s)\right)\,ds\right]\\
&= \int_0^\infty\mathbf E\left[\mathbf 1_{\{s<\tau_{0}(t)\}} \int_{-X^{t,(\rho)}(s)}^\infty \phi\left(-X^{t,(\rho)}(s)-r,B^x(s\wedge \tau_\Omega(x))\right)\rho\left(-X^{t,(\rho)}(s),r\right)\,dr\right]ds
\end{align*}
Next,  using the independence of $-X^{t,(\rho)}(s\wedge\tau_0(t))$ and $B^x(s\wedge \tau_\Omega(x))$, $\{s<\tau_{0}(t)\}=\{0<-X^{t,(\rho)}(s)\}$, Fubini's Theorem and standard change of variables, we obtain
\begin{align*}
 u(t,x)=&\int_\Omega \int_0^\infty\left(\int_0^t  \left( \int_{z}^\infty \phi\left(z-r,y\right)\rho(z,r)\,dr\right)p^{(\rho)}_s(t,z)\,dz\right)p^\Omega_s(x,y)\,ds\,dy\\
 =&\, \int_\Omega \int_0^\infty\left(\int_0^t  \left( \int_{-\infty}^0 \phi\left(r,y\right)\rho(z,z-r)\,dr\right)p^{(\rho)}_s(t,z)\,dz\right)p^\Omega_s(x,y)\,ds\,dy\\
 =&\,  \int_{-\infty}^0\int_\Omega \phi\left(r,y\right)\left( \int_0^t  \rho(z,z-r)\int_0^\infty\,p^{(\rho)}_s(t,z)p^\Omega_s(x,y)\,ds\,dz\right)dy\,dr.
\end{align*}
By a density argument the identity (\ref{SRin}) holds  for every $\phi \in B((-\infty, n^{-1})\times\Omega)\cap C((-\infty, n^{-1})\times\Omega)$ for every $n\in\mathbb N$. Considering the non-negative increasing sequence $\phi_n =\mathbf 1_{(-\infty,n^{-1})\times\Omega}$, $n\in\mathbb N$, by Monotone Convergence Theorem one can pass to the limit in both sides of (\ref{SRin}),  confirming that  $H_{t,x}$ induces a finite measure on $(-\infty,0)\times\Omega$, as the right hand side of (\ref{SRin}) is finite. By another density argument the equality (\ref{SRin}) holds for every $\phi\in C_\infty((-\infty,0)\times \Omega)\cap\{f(0-)=f(x)=0,\, x\in\partial\Omega\}$, and we are done by Riesz-Markov-Kakutani representation Theorem \cite[Theorem 1.7.3]{KV0}.
\end{proof}

\begin{remark}\label{rmk:lpconv2}
Suppose that \ref{H0} and \ref{H1} hold, and that $\phi_n,\phi \in B((-\infty,0]\times\overline\Omega)$, for $n\in\mathbb N$, such that $\phi_n\to\phi$ a.e. on $(-\infty,0]\times\overline\Omega$, $\sup_n\|\phi_n\|_{B((-\infty,0]\times\overline\Omega)}<\infty$,  and $f\in B((0,T]\times\Omega)$. Then Proposition \ref{prop:HK} and Dominated Convergence Theorem imply that $u_n\to u$  pointwise on $(0,T]\times\Omega$ and $\sup_n\|u_n\|_{B((-\infty,0]\times\overline\Omega)}<\infty$. Here $u_n$ is defined as \eqref{thm:Caputo_SR} for $\phi_n, f$,  $n\in\mathbb N$, and $u$ is defined as \eqref{thm:Caputo_SR} for $\phi, f$. This in turn implies the convergence $u_n\to u$ in $L^p((0,T)\times\Omega)$ for each $p\in[1,\infty)$.
\end{remark}

\section{Stochastic representation for solutions in weak sense}\label{sec:weak}
In Section \ref{sec:generaltheory},  the stochastic representation of the solution to the nonlocal-in-time evolution model \eqref{preRL} is established
in case that the data is smooth and compatible. The aim of this section is to show that
the representation \eqref{thm:Caputo_SR} still provides a solution of \eqref{preRL} in the weak sense,
even though the data does not satisfies the smoothness and compatibility conditions required in Section \ref{sec:generaltheory}.
Now we denote by  $W^{1,p}(\Omega)$ the standard Sobolev space of $p$-integrable functions on $\Omega$ with
$p$-integrable weak first derivatives, $p\ge1$. Denote by $H^{-1}(\Omega)$ the dual of $H_0^{1}(\Omega)$,
where $H_0^{1}(\Omega)$ is the closure of $C_c^\infty(\Omega)$ in $W^{1,2}(\Omega)$.

In case that the kernel is time-independent, i.e., $\rho(t,r)\equiv\rho(r)$, the existence and uniqueness  of the weak solution \eqref{DYZws} has been confirmed in \cite{DYZ17}.
The uniqueness argument for the more general variables-separable kernel $\rho(t,r)=p(r)q(r)$ is similar, so we only present some useful results here and omit some similar detailed proof in order to avoid redundancy. 
We do not prove uniqueness of weak solutions for our general time-dependent kernel  $\rho(t,r)$.

\begin{lemma}\label{lem:inte-p}
Suppose that $u \in B(-\infty,T) \cap L^1(-\infty, T)$, and $v \in C_c^\infty(0, T)$ with zero extension out of the interval $(0,T)$. Further, we suppose that
%\footnote{\red{Don't you mean $$
 %\int_{0}^T| D_{\infty}^{(\rho)} u(t) |\,dt\le\int_{0}^T  \int_{0}^\infty |u(t)-u(t-r)| \rho(t,r) \,dr \,dt=\int_{0}^T  \int_{-\infty}^t |u(t)-u(r)| \rho(t,t-r) \,dr \,dt<\infty\,?$$}}
\begin{equation}\label{eqn:domin}
    \int_{0}^T  \int_{0}^\infty |u(t)-u(t-r)| \rho(t,r) \,dr \,dt < \infty.
\end{equation}
Then  it holds that
\begin{equation*}
     \int_{0}^T D_{\infty}^{(\rho)} u(t) v(t)\,dt =- \int_{-\infty}^T u(t) (D_{\infty}^{(\rho),*} v) (t)\,dt
\end{equation*}
with
\begin{equation}\label{eqn:Gdstar}
 D_{\infty}^{(\rho),*} v(t) = - \int_0^{\infty} v(t) \rho(t,r)-v(t+r)\rho(t+r,r) \,dr.
\end{equation}
\end{lemma}

The next lemma gives an upper bound of $\Gd$ for smooth functions in Sobolev spaces.
\begin{lemma}\label{lem:bddG-r}
Let the kernel $\rho$ satisfy \ref{H0}. Then the operator $\Gd$  defined by \eqref{eqn:op} satisfies
$$  \| \Gd v\|_{L^p(-\infty,T)} \le C   \|v\|_{W^{1,p}(-\infty,T)},\quad v\in W^{1,p}(-\infty,T).  $$
with $p\in[1,\infty]$.
\end{lemma}
\begin{proof}
We only prove the result for  $p\in[1,\infty)$, as  the case  $p=\infty$ follows analogously.
By H\"older's inequality and assumption \ref{H0} we have that for $p\in(1,\infty)$
\begin{align*}
&\qquad\int_{-\infty}^T \left(\int_{0}^1 |u(t)-u(t-r)| \rho(t,r) \,ds \right)^p dt \\
&\le  \int_{-\infty}^T  \int_{0}^1 \frac{|u(t)-u(t-r)|^p} {r^p} r\rho(t,r)  \,dr
\left(\int_{0}^1 r\rho(t,r)\,dr\right)^{p-1} dt  \\
&\le c  \int_{-\infty}^T  \int_{0}^1 \frac{|u(t)-u(t-r)|^p} {r^p} r\rho(t,r)  \,ds \,dt \\
&\le  c \int_{0}^1 {r^{1-p}}  |\max_{t}\rho(t,r)| \int_{-\infty}^T  |u(t)-u(t-r)|^p \,dt \,dr \\
&\le  c \int_{0}^1 r |\max_{t}\rho(t,r)|  \,dr \|  u \|_{W^{1,p}(-\infty,T)}^p \le c\|  u \|_{W^{1,p}(-\infty,T)}^p,
\end{align*}
where we apply the fact that
$\int_{-\infty}^T  |u(t)-u(t-r)|^p \,dt \le c |r|^p \| u \|_{W^{1,p}(-\infty,T)}^p$ in the second last inequality.
%Then applying the fact that
%$\int_{-\infty}^T  |u(t)-u(t-s)|^p \,dt \le c |s|^p \| u \|_{W^{1,p}(-\infty,T)}^p$,
%we arrive at
%\begin{align*}
%\int_{-\infty}^T \left(\int_{0}^1 |u(t)-u(t-s)| \rho(t,s) \,ds \right)^p \,dt
%&\le  c \int_{0}^1 {s^{1-p}}  |\max_{t}\rho(t,s)| \int_{-\infty}^T  |u(t)-u(t-s)|^p \,dt \,ds \\
%&\le  c \int_{0}^1 s |\max_{t}\rho(t,s)|  \,ds \|  u \|_{W^{1,p}(-\infty,T)} \le c\|  u \|_{W^{1,p}(-\infty,T)}^p.\\
%\end{align*}
On the other hand, we have the following estimate
\begin{align*}
&\qquad\int_{-\infty}^T \left(\int_{1}^\infty |u(t)-u(t-r)| \rho(t,r) \,dr \right)^p dt \\
&\le  \int_{-\infty}^T  \int_{1}^\infty  |u(t)-u(t-r)|^p \rho(t,r)  \,dr
\left(\int_{1}^\infty   \rho(t,r) \,dr \right)^{p-1} dt \\
&\le c  \int_{-\infty}^T  \int_{0}^1  |u(t)-u(t-r)|^p \rho(t,r)  \,dr \,dt \\
&\le  c \int_{1}^\infty    \max_{t}\rho(t,r) \int_{-\infty}^T  |u(t)-u(t-r)|^p \,dt \,dr \\
&\le  c \int_{1}^\infty   \max_{t}\rho(t,r)  \,dr \|  u \|_{L^{ p}(-\infty,T)}^p \le c\|  u \|_{W^{1,p}(-\infty,T)}^p.\\
\end{align*}
Then we obtain the desired assertion.
\end{proof}

Similar argument yields the following a priori bound for the dual operator $D_{\infty}^{(\rho),*}$ given by \eqref{eqn:Gdstar}.
\begin{lemma}\label{lem:bddG-r2}
Let the kernel $\rho$ satisfy \ref{H0} and let the operator $D_{\infty}^{(\rho),*} $ be defined by \eqref{eqn:Gdstar}. Then for any $v\in {W^{1,p}(\mathbb{R})}$ with $p\in[1,\infty]$, it holds that
$$  \| D_{\infty}^{(\rho),*} v\|_{L^p (\mathbb{R})} \le C   \| v \|_{W^{1,p}(\mathbb{R})}.$$
%and further
%$$  \| \Gd(\fy)\|_{L^\infty(0, T)} \le C   \| \fy' \|_{L^\infty(-\delta,T)},\quad   \fy\in W^{1,\infty}(-\delta,T).    $$
\end{lemma}
\begin{proof}
First, we use the following splitting
\begin{equation*}
 D_{\infty}^{(\rho),*} v(t) = \int_0^{\infty} ( v(t+r) - v(t))\rho(t,r) \,dr + \int_0^{\infty}  v(t+r)(\rho(t,r)-\rho(t+r)) \,dr=I_1+I_2.
\end{equation*}
Now using the same argument as that in Lemma \ref{lem:bddG-r}, we derive that for $p\in[1,\infty)$
$$  \| I_1\|_{L^p (\mathbb{R})} \le C   \| v \|_{W^{1,p}(\mathbb{R})}.$$
Therefore it suffices to  bound  $I_2$. For  $p\in[1,\infty)$, by H\"older's inequality and assumption \ref{H0} we have that
\begin{align*}
&\int_{-\infty}^\infty \left(\int_{0}^1 |v(t+r)| |\rho(t,r)-\rho(t+r)|\,dr \right)^p \,dt\\
\le& \int_{-\infty}^\infty \int_{0}^1  |v(t+r)|^p |\rho(t,r)-\rho(t+r)|\,dr \left(\int_{0}^1 |\rho(t,r)-\rho(t+r,r)|\,dr\right)^{p-1} \,dt.
\end{align*}
Then we observe that
\begin{align*}
 \int_{0}^1 |\rho(t,r)-\rho(t+r,r)|\,dr \le  \int_{0}^1 \int_t^{t+r}| \partial_y\rho(y,r)|\,dr \le \int_{0}^1 r \max_{t}| \partial_t\rho(t,r)|\,dr \le c,
 \end{align*}
and hence
\begin{align*}
&\qquad\int_{-\infty}^\infty \left(\int_{0}^1 |v(t+r)| |\rho(t,r)-\rho(t+r,r)|\,dr \right)^p \,dt\\
&\le  c\int_{-\infty}^\infty \int_{0}^1  |v(t+r)|^p |\rho(t,r)-\rho(t+r,r)|\,dr   \,dt\\
&\le  c \int_{0}^1 \int_{-\infty}^\infty|v(t+r)|^p \,dt \max_t|\rho(t,r)-\rho(t+r,r)|\,dr \\
&\le c\| v \|_{L^p(\mathbb{R})} \int_{0}^1 r \max_t|\partial_t\rho(t,r)|\,dr \le c \| v \|_{L^p(\mathbb{R})}.
\end{align*}
Meanwhile, applying the following observation
\begin{align*}
 \int_{1}^\infty |\rho(t,r)-\rho(t+r,r)|\,dr \le  \int_{1}^\infty |\rho(t,r)|+|\rho(t+r,r)|\,dr\le c,
\end{align*}
we have the following estimate
\begin{align*}
&\qquad\int_{-\infty}^\infty \left(\int_{1}^\infty |v(t+r)|  |\rho(t,r)-\rho(t+r,r)|\,ds \right)^p \,dt\\
&\le  c\int_{-\infty}^\infty \int_{1}^\infty  |v(t+r)|^p |\rho(t,r)-\rho(t+r,r)|\,ds   \,dt\\
&\le  c \int_{1}^\infty \int_{-\infty}^\infty|v(t+r)|^p \,dt (|\rho(t,r)|+|\rho(t+r,r)|)\,dr \\
& \le c \| v \|_{L^p(\mathbb{R})},
\end{align*}
which yields that
$$  \| I_2\|_{L^p (\mathbb{R})} \le C   \| v \|_{W^{1,p}(\mathbb{R})}.$$
This completes the proof for $p\in[1,\infty)$, and  the case  that $p=\infty$ follows analogously.
\end{proof}

Then we have the following result for a smooth function with compact support.
\begin{corollary}\label{cor:bddG-r2}
Let the kernel $\rho$ satisfy \ref{H0} and let the operator $D_{\infty}^{(\rho),*} $ be defined by \eqref{eqn:Gdstar}. Then $ D_{\infty}^{(\rho),*} v\in L^1(-\infty,T)\cap L^\infty(-\infty,T)$
for any $v\in  C_c^1(0,T)$.
\end{corollary}

\begin{definition}\label{def:fwd}
We define the \emph{weak Marchaud-type derivative}  of a function $u\in L_{loc}^1(\mathbb{R};B)$, for  a Banach space $B$, to be a  function $\widetilde{D_{\infty}^{(\rho)} }u \in L_{loc}^1(\mathbb{R};B)$ that satisfies
$$\int_{\mathbb R} \widetilde{D_{\infty}^{(\rho)} }u(t) v(t) \,dt = \int_{\mathbb R} u(t) (D^{(\rho),*}_{\infty}v)(t) \,dt,\quad \text{for every }v\in C_c^\infty(0,T), $$
with the integral defined in the Bochner sense.
\end{definition}

The following Lemma gives the equivalence between the variational nonlocal operator and the strong one in the case that $B=\mathbb{R}$ and $\rho$ is variables-separable.

\begin{lemma}\label{lem:density2}
Suppose that  the kernel $\rho$ satisfies \ref{H0} and it is variables-separable, i.e., $\rho(t,r)=p(t)q(r)$ with $p(t) \in C^1[0,T]$ and $p(t)\ge c_1>0$.
Moreover, we let $u\in L^\infty(\mathbb{R}) $ and $ \widetilde{D_{\infty}^{(\rho)}} u\in L^2(0, T)$.
Then $\Gd u\in L^2(0,T)$ and
$$\Gd u=\widetilde \Gd u \quad \text{almost everywhere},$$
where $\Gd$ is defined by \eqref{eqn:op}.
\end{lemma}

\begin{proof}
First of all, we consider the case that the kernel function is translation preserved, i.e., $\rho(t,r)=\rho(r)$.
To this end, we define the truncated nonlocal operator
$$  D_{\delta}^{(\rho)} u(t) = \int_0^\delta (u(t)-u(t-r))\rho(r) \,dr$$
as well as its adjoint operator $D_{\delta}^{(\rho),*}$ and the weak operator $\widetilde{D_{\delta}^{(\rho)}}$.
Since for any $\delta>0$, we have
\begin{align*}
\int_\delta^{\infty} (u(t) - u(t-r)) \rho(r) \,dr
= u(t) \int_\delta^\infty \rho(r) \,dr - \int_\delta^ \infty u(t-r) \rho(r) \,dr \in L^2(0,T),
\end{align*}
by assumption \ref{H0}. By the definition of the weak operator, one may deduce that
$$\widetilde{D_{\delta}^{(\rho)}} u (t)= \widetilde{D_{\infty}^{(\rho)}} u(t) - \int_\delta^{\infty} (u(t) - u(t-r)) \rho(r) \,dr \in L^2(0,T)$$
%Then for any $v\in C_c^\infty(0,T)$,
%\begin{equation*}
%\langle\widetilde{D_{\delta}^{(\rho)} }u, v\rangle_{0}^T = \langle u,  D_{\delta}^{(\rho),*} v\rangle_{-\infty}^T
%=  \langle u,  D_{\infty}^{(\rho),*} v\rangle _{-\infty}^T -
%\end{equation*}
%
%by \cite[Lemma 2.4]{DYZ17} and letting $\delta\to\infty$,  we have
%$$\langle\widetilde{D_{\delta}^{(\rho)} }u, v\rangle_{0}^T = \langle u,  D_{\delta}^{(\rho),*} v\rangle_{-\infty}^T \to  \langle u,  D_{\infty}^{(\rho),*} v\rangle _{-\infty}^T
%= \langle\widetilde{D_{\infty}^{(\rho)}} u,  v\rangle_{0}^T, $$
%which means $\widetilde{D_{\delta}^{(\rho)} }u \to \widetilde{D_{\infty}^{(\rho)}} u$ weakly in $L^2(0,T)$. Then  the principle of uniform boundedness leads to
%\begin{equation*}
%\sup_{\delta} \|  \widetilde{D_{\delta}^{(\rho)} }u  \|_{L^2(0,T)} < C.
%\end{equation*}
Now by Lemma \cite[Lemma 2.4]{DYZ17} we have that $D_{\delta}^{(\rho)} u \in L^2(0,T)$ and $D_{\delta}^{(\rho)} u =\widetilde{D_{\delta}^{(\rho)} }u $.
As a result, we derive that
\begin{align*}
\Gd u(t) = \int_0^{\infty} (u(t) - u(t-r)) \rho(r) \,dr
=D_{\delta}^{(\rho)}u(t)  + \int_\delta^{\infty} (u(t) - u(t-r)) \rho(r) \,dr  \in L^2(0,T),
\end{align*}
and hence $\Gd u = \widetilde \Gd u$ almost everywhere.

Next, we consider the case that $\rho(t,r) = p(t) q(r)$ and define the operator
\begin{equation*}
 D_{\infty}^{(q)} u(t) = \int_0^\infty (u(t)- u(t-r)) q(r) \,ds.
\end{equation*}
The same as before, we may define corresponding adjoint and weak operators. Define $\langle f,g \rangle_a^b:=\int_a^b fg\,dt$, $b>a\ge-\infty$.
Then we note that
\begin{equation*}
\langle p\widetilde{D_{\infty}^{(q)} }u, v\rangle_{0}^T = \langle u,  D_{\infty}^{(q),*} (pv)\rangle_{-\infty}^T =   \langle u,  D_{\infty}^{(\rho),*} v\rangle_{-\infty}^T = \langle \widetilde{D_{\infty}^{(\rho)}} u,  v\rangle_{0}^T,
\end{equation*}
which together with the positivity assumption on $p(t)$ yields that
\begin{equation*}
  \widetilde{D_{\infty}^{(q)} }u(t) = \frac1{p(t) } \widetilde{D_{\infty}^{(\rho)}} u(t) \le \frac{1}{c_1}\left | \widetilde{D_{\infty}^{(\rho)}} u(t)\right| \in L^2(0,T).
\end{equation*}
As a result, we obtain that $D_{\infty}^{(q)}u(t) =  \widetilde{D_{\infty}^{(q)} }u(t)  \in L^2(0,T)$ and
 $$  \Gd u(t) =  {p(t) } \widetilde{D_{\infty}^{(q)}}u(t)  =\widetilde{\Gd}u(t)\in L^2(0,T).$$
\end{proof}

\begin{lemma}\label{lem:density3}
Let $\tilde u\in B((-\infty,T]\times\Omega)$ be the function defined in (\ref{thm:Caputo_SR}) under the assumptions \ref{H0} and \ref{H1}, for $\phi\in L^{\infty}(-\infty,0; H^1_0(\Omega))$ and $f\in L^{\infty}(0,T; H^1_0(\Omega))$.  Then $\tilde u\in L^{\infty}(-\infty,T; H^1_0(\Omega))$.
\end{lemma}
\begin{proof} Consider (\ref{thm:Caputo_SR}) for  $f=0$ (the proof for $f\neq 0$ is similar and omitted). Fix $t>0$.  By \cite[Chapter 7.1]{Evans} we have $T^\Omega_s\phi(r,\cdot) =\mathbf E[\phi(r,B^\cdot(s))\mathbf 1_{\{s<\tau_{\Omega}\}}]\in H^1_0(\Omega)$ for a.e. $r\in (-\infty,0)$ and $s\ge0$. Consider the Borel probability space $(\Gamma,\mu_t)$, where $\Gamma =(-\infty,0)\times(0,\infty)$ and $\mu_{t}(dsdr)=\left(\int_0^t\rho(z,z-r) p_s^{(\rho)}(t,z)\,dz\right)dsdr$, so that formula \eqref{SRin} reads $u(t,x)=\int_{\Gamma} T^\Omega_s\phi(r,x)\, \mu_t(dsdr)$.  Note that for a.e. $r\in(-\infty, 0)$ and every $s\ge 0$
$$
\|T^\Omega_{s}\phi(r)\|_{H^1(\Omega)} \le \|\phi(r)\|_{H^1(\Omega)}\le  \|\phi\|_{L^{\infty}(-\infty,0; H^1_0(\Omega))}=:C,
$$
where the first inequality holds by \cite[Chapter 7.1, Theorem 5.(i)]{Evans}, as  $\phi(r)\in H^1_0(\Omega)$ for a.e. $r\in(-\infty,0)$.
We conclude that $\tilde u(t)\in H^1_{0}(\Omega)$, because the above bound proves that $T^\Omega_{\cdot}\phi(\cdot):(\Gamma,\mu_t) \to H^{1}_0(\Omega)$ is Bochner integrable, which implies that $\tilde u(t)=\int_{\Gamma}T^\Omega_{\cdot}\phi(\cdot)\,\mu_t(d\cdot)=\lim_{n\to\infty}S_n$ in $H^1(\Omega)$, where each $S_n$ is a linear combination of functions in $H^1_{0}(\Omega)$.\\
Formula \eqref{SRin} suggests the definition
\begin{align*}
\nabla \tilde u(t,x) :&= \int_{-\infty}^0 \left(\int_0^t \rho(z,z-r)\left(\int_0^\infty \nabla T^\Omega_s\phi(r,x)p_s^{(\rho)}(t,z)\,ds\right)dz\right)dr\\
&=\int_{\Gamma} \nabla T^\Omega_s\phi(r,x)\, \mu_t(dsdr).
\end{align*}
Then $\nabla \tilde u(t)\in L^2(\Omega)$, because
\begin{align*}
\int_\Omega \left(\nabla \tilde u(t,x)\right)^2\,dx&=\int_\Omega\left( \int_{\Gamma} \nabla T^\Omega_s\phi(r,x)\, \mu_t(dsdr)\right)\left( \int_{\Gamma} \nabla T^\Omega_{s'}\phi(r',x)\, \mu_t(ds'dr')\right)dx\\
&=\int_{\Gamma}\int_{\Gamma}\left(  \int_\Omega \nabla T^\Omega_s\phi(r,x)\nabla T^\Omega_{s'}\phi(r',x) \,dx \right)\mu_t(dsdr)\mu_t(ds'dr')\\
&\le\int_{\Gamma}\int_{\Gamma}\| T^\Omega_s\phi(r)\|_{H^1(\Omega)}\| T^\Omega_{s'}\phi(r')\|_{H^1(\Omega)}\,\mu_t(dsdr)\mu_t(ds'dr')\\
&\le C^2 \left(\int_{\Gamma}\mu_t(dsdr)\right)^2=C^2.
\end{align*}

 Applying Fubini's Theorem to the definition of weak derivative proves that $\nabla \tilde u(t)$ is indeed the weak derivative of $\tilde u(t)$. Finally, $\sup_{t\in(0,T)}\int_\Omega \left(\nabla \tilde u(t,x)\right)^2\,dx\le C^2$ and the smoothness of $\phi$ implies that $ \tilde u\in L^{\infty}(-\infty,T; H^1_0(\Omega))$, concluding the proof.
\end{proof}

Next we shall show that the stochastic representation \eqref{thm:Caputo_SR} provides the weak solution of problem \eqref{preRL},
whose definition is given as below.

\begin{definition}\label{def:DYZws}
A function $ u$ is called a \emph{weak solution to problem } \eqref{preRL}  if $ u \in L^2(0,T;H_0^1(\Omega))$ and $\widetilde {D_{\infty}^{(\rho)}} u \in L^2(0,T; H^{-1}(\Omega))$,
and for every $v \in L^2(0,T;H_0^1(\Omega))$ (with zero extension to $t<0$)%\footnote{\red{Define pairing $\langle ,\rangle_\Omega$.}}
\begin{equation}\label{DYZws}
\left\{\begin{split}
\langle \widetilde {D_{\infty}^{(\rho)}}  u,v \rangle  &= -\langle \nabla   u,\nabla v \rangle+ \langle f, v \rangle, &\text{and},\\
 u(t) &= \phi(t), &\text{for a.e. } t\in (-\infty,0),
\end{split}\right.
\end{equation}
where  the notation $\langle\cdot,\cdot\rangle$ is defined by
$$
\langle u,v\rangle=\int_{-\infty}^T\int_\Omega u(t,x)v(t,x)\,dx\,dt,
$$
or the duality in case that $u\in L^2(0,T;H^{-1}(\Omega))$.
\end{definition}

\begin{remark}\label{rem:weak}
If $u$ is the weak solution of \eqref{preRL} and $\widetilde {D_{\infty}^{(\rho)}} u \in L^2(0,T;L^2\II)$, we have $\widetilde{D_{\infty}^{(\rho)}} u =D_{\infty}^{(\rho)} u $
by Lemma \ref{lem:density2}, provided that
the kernel function is variables-separable, i.e., $\rho(t,s)=p(t)q(s)$ with $p(t) \in C^1[0,T]$ and $p(t)\ge c_1>0$. Then $u$ satisfies the equation \eqref{preRL} almost everywhere.
\end{remark}

\begin{theorem}\label{thm:main}
Assume \ref{H0} and \ref{H1}. %Moreover, let the kernel function be variables-separable, i.e., $\rho(t,s)=p(t)q(s)$ with $p(t) \in C^1[0,T]$ and $p(t)\ge c_1>0$.
Let $u$ be given by formula \eqref{thm:Caputo_SR}, where $\phi\in L^{\infty}(-\infty,0; H^1_0(\Omega))\cap L^\infty((-\infty,0)\times\Omega)$ and $f\in L^{\infty}(0,T; H^1_0(\Omega))\cap L^\infty((0,T)\times\Omega)$. Define the  extension $\tilde u$ of $u$  as
	\begin{equation}\label{ext}
	\tilde u:=\left\{\begin{split}
	u,& &\text{on }&(0,T]\times  \Omega,\\
	\phi,& &\text{on }& (-\infty,0)\times  \Omega.
	\end{split}
	\right.
	\end{equation}
	Then $\tilde u$ is a weak solution to problem (\ref{preRL}).
%\red{(i) add uniqueness remark for time variable kernel and  clarify connection with definition of weak solution, (ii) remark on extension to $\Delta\mapsto \mathcal L_\Omega$ Hilbert-Schmidt operators?  }
%Also (iii) presence of finite horizon/support of $\rho$ is relevant for uniqueness.
\end{theorem}

\begin{proof} Assume for the first two steps  that $\phi$  satisfies \ref{H2}.\\
\emph{Step 1}:   Let $u$ be a solution in the domain of the generator to problem (\ref{postRL}) for $g\equiv  f+f_\phi$, and initial condition $\phi(0)$, for some $f\in C_{\partial\Omega}([0,T]\times\Omega)$. As $u\in \text{Dom}(\mathcal L_\Omega^{(\rho)})$,
 by Lemma \ref{lem32}, $u-\phi(0)\in\text{Dom}(\mathcal L_\Omega^{(\rho),\text{kill}})$, and hence applying Corollary \ref{cor32}-(i) there exists
$\{\hat u_n\}_{n\in\mathbb N}\subset \text{Dom}(\mathcal L_\Omega^{(\rho),\text{kill}})$
such that
$$\hat u_n\to u-\phi(0),~~ \mathcal L_\Omega^{(\rho)}\hat u_n\to \mathcal L_\Omega^{(\rho)} (u-\phi(0))~~\text{and} ~~\mathcal L_\Omega^{(\rho)}\hat u_n=(-D_0^{(\rho)}+\Delta)\hat u_n.$$
Then we apply Lemma \ref{lem32} and Lemma \ref{lem:31} to obtain that
$$u_n:=\hat u_n+\phi(0)\in  \text{Dom}(\mathcal L_\Omega^{(\rho)}), ~~u_n\to u, ~~
\mathcal L_\Omega^{(\rho)}u_n=\mathcal L_\Omega^{(\rho)}\hat u_n+\Delta\phi(0)\to \mathcal L_\Omega^{(\rho)} u$$
and $u_n(0)=\phi(0)$ for all $n\in\mathbb N$. Then using the fact that  $D^{(\rho)}_\infty \tilde u_n = D^{(\rho)}_0u_n-f_\phi$  for $t\in[0,T]$, we have
\begin{equation*}%\label{eq:extension}
 (D^{(\rho)}_0 - \Delta) u_n - f_\phi = \Gd \tilde u_n - \Delta \tilde u_n,\quad \text{on }[0,T]\times\Omega,
\end{equation*}
where $ \tilde u_n $ is defined for each $n\in\mathbb N$ by
	\begin{equation}\label{eqn:extun}
	\tilde u_n :=\left\{\begin{split}
	u_n,& &\text{on }&(0,T]\times  \Omega,\\
	\phi,& &\text{on }& (-\infty,0]\times  \Omega.
	\end{split}
	\right.
	\end{equation}
Therefore, we have that
\begin{align*}
  (-D_{\infty}^{(\rho)}+\Delta)\tilde u_n
=  (-D^{(\rho)}_0+\Delta)u_n+f_\phi  \rightarrow  \mathcal L_\Omega^{(\rho)} u+f_\phi =   - f,
\end{align*}
where the convergence is in $C_{\partial\Omega}([0,T]\times\Omega)$.

On the other hand, we apply Corollary \ref{cor:bddG-r2} for any $ v \in C_c^\infty((0,T)\times \Omega)$ to obtain as $n\to\infty$
\begin{equation*}
  \langle (-D_{\infty}^{(\rho)}+\Delta)\tilde u_n , v \rangle
=  \langle \tilde u_n , (-D_{\infty}^{(\rho),*}+\Delta) v \rangle \rightarrow \langle \tilde u , (-D_{\infty}^{(\rho),*}+\Delta) v \rangle,
\end{equation*}
where Corollary \ref{cor:bddG-r2} guarantees that $(-D_{\infty}^{(\rho),*}+\Delta) v \in L^1((-\infty,0)\times\Omega)\cap L^\infty ((0,T)\times\Omega)$,
and hence
\begin{equation*}
\langle u, (D_{\infty}^{(\rho),*}  - \Delta) v\rangle = \langle f, v\rangle,\qquad \text{for any} ~~v \in C_c^\infty((0,T)\times \Omega).
\end{equation*}

%I know little about Bochner integration. Do you mean that $\int_0^1\|f(s)\|_{H^1_0(\Omega)}ds<\infty$ implies that $\exists f_n\in \text{Simple}(0,1;H^1_0(\Omega))$ such that  $\|f_n(s)-f(s)\|_{H^1} \to 0$ for a.e. $s\in(0,1)$, and $\|f_n(s)\|_{H^1}\le g(s)\in L^1(0,1)$, implying  $H_0^1(\Omega)\ni\int_0^1f_n(s)ds\stackrel{H^1}{\to}\int_0^1f(s)ds$?

\emph{Step 2}:  Let now $u$ be the generalized solution to problem (\ref{postRL}) for $g=f+f_\phi$, where $f\in L^\infty((0,T)\times\Omega)$,
and let $\tilde u$ be its extension with historical initial data $\phi$.
By the definition of the generalized solution, we
pick a sequence $f_n\in C_{\partial\Omega} ([0,T]\times\Omega) $ such that
$$f_n\to f\quad  \text{a.e.},\qquad f_n(0)=-(f_\phi(0)+\Delta\phi(0))\qquad \text{and}\qquad \sup_n\|f_n\|_\infty<\infty.$$
Besides, we denote by $u_n$ the respective solution in the domain of the generator and let $\tilde u_n$ be its extension by \eqref{eqn:extun}.
Then by Step 1, we know that each $\tilde u_n$ satisfies
$$
\langle \tilde u_n , (-D_{\infty}^{(\rho),*}+\Delta) v \rangle
=\langle -f_n , v \rangle, \qquad \text{for any}\quad v\in   C_c^\infty((0,T)\times \Omega),
$$
as well as the initial and boundary conditions in \eqref{preRL}. Now the Dominated Convergence Theorem provided the uniform upper bound of $f_n$ implies that
$$ f_n \rightarrow f \qquad \text{in}\qquad L^2(0,T; L^2\II)~~\qquad \text{as}\quad n\rightarrow \infty.  $$
On the other hand, we have
%\footnote{ \red{$L^2(-\infty,T; L^2\II)$ convergence does not hold, because $\phi$ need not be square integrable.
%Instead use convergence on $\tilde u_n \to \tilde u$ p.w. on $(-\infty,T]\times \overline \Omega)$ and $\sup_n\|\tilde u_n\|_{C((-\infty,T]\times \overline \Omega)}<\infty$ and by Lemma \ref{lem:bddG-r2} $ (D_{\infty}^{(\rho),*}-\Delta) v\in L^1((-\infty,T]\times \overline \Omega))$}}
$\tilde u_n \to \tilde u$ in $L^2(0,T; L^2\II)$ by Remark \ref{rmk:lpconv1}.
Meanwhile $(D_{\infty}^{(\rho),*}-\Delta) v\in L^1((-\infty,0)\times\Omega)\cap L^\infty ((0,T)\times\Omega)$ for any $ v \in C_c^\infty((0,T)\times \Omega)$ by Corollary \ref{cor:bddG-r2}.
Therefore we obtain as $n\to\infty$ %And so for any  $ v \in C_c^\infty((0,T)\times \Omega)$, we obtain as $n\to\infty$
\begin{equation*}
\langle \tilde u_n , (D_{\infty}^{(\rho),*}-\Delta) v \rangle \rightarrow
\langle \tilde u , (D_{\infty}^{(\rho),*}-\Delta) v \rangle,\qquad \text{for any}\quad v\in   C_c^\infty((0,T)\times \Omega).
\end{equation*}

\emph{Step 3}:
Now we consider the case that $\phi\in  L^{\infty}(-\infty,0; H^1_0(\Omega))\cap L^\infty((-\infty,0)\times\Omega))$
and $f\in L^{\infty}(0,T; H^1_0(\Omega))\cap L^\infty((0,T)\times\Omega)$.
To this end, we set functions $\phi_K(t,x)=\phi(t,x) \mathbf 1_{\{t<-K\}}$,
for $K\in\mathbb N$. By the density of $\text{Span}\{C^\infty_c(-K,0)\cdot C^\infty_c(\Omega)\}$ in $B([-K,0]\times\overline\Omega)$ with respect to sequential convergence a.e.,  we choose $\phi_{K,j}\in\text{Span}\{C^\infty_c(-K,0)\cdot C^\infty_c(\Omega)\}$
such that
$$\phi_{K,j}\to \phi_K\quad  \text{a.e.}\quad  \text{and}\quad  \sup_j\|\phi_{K,j}\|_\infty<\infty.$$
By Remark \ref{rem:H2}-(i), we know that $\phi_{K,j}$  satisfies assumption \ref{H2} for each $j\in\mathbb N$.
Denote by $u_{K,j}$ the generalized solution with the initial data $\phi_{K,j}$ and source term $f$,
and denote by $u_K$ the function given by formula (\ref{thm:Caputo_SR}) with $\phi\equiv \phi_K$ and source term $f$.
By Remark \ref{rmk:lpconv2} we conclude that
$$\sup_j\|\tilde u_{K,j}\|_\infty<\infty\quad \text{and}\quad \tilde u_{K,j}\rightarrow\tilde u_K\quad \text{a.e. on} ~~(-K,T]\times\Omega.$$
Then for any $ v \in C_c^\infty((0,T)\times \Omega)$,
we know that $(D_{\infty}^{(\rho)}-\Delta)^*v\in L^1((-K,0)\times\Omega)\cap L^\infty((0,T)\times\Omega)$  by  Corollary \ref{cor:bddG-r2}, and hence
\begin{equation}\label{almost_finished}
  \langle  \tilde u_K , (D_{\infty}^{(\rho),*}-\Delta)v \rangle=\lim_{j\to\infty}\langle  \tilde u_{K,j} , (D_{\infty}^{(\rho),*}-\Delta)v \rangle =\langle  f ,   v \rangle,
\end{equation}
and $\tilde u_K = \phi_K$ on $(-K,0]\times\Omega$.
We can now  pass to the limit as $K\to\infty$ in \eqref{almost_finished}, given that $\tilde u_K\to \tilde u$ a.e. on $(-\infty,T)\times\Omega$, with $\sup_K\|\tilde u_{K}\|_\infty<\infty$, again by Remark \ref{rmk:lpconv2}, and $(D_{\infty}^{(\rho)}-\Delta)^*v\in  L^1((-\infty,0)\times\Omega)\cap  L^\infty((0,T)\times\Omega)$ by Corollary \ref{cor:bddG-r2}. Here $u$ is defined by \eqref{thm:Caputo_SR} for $\phi$ and $f$, and $\tilde u$ by \eqref{ext}.
Therefore
$$
 \langle  \tilde u , (D_{\infty}^{(\rho),*}-\Delta)v \rangle
 = \langle f , v \rangle.
$$
 By Lemma  \ref{lem:density3} and the smoothness of the problem data $f$ and $\phi$ we obtain $\tilde u\in L^2(0,T; H^1_0(\Omega))$ and $\tilde u $ satisfies the identities in \eqref{DYZws}. Also,  for every $w\in H^1_0(\Omega)$, $v\in C_c^1(0,T)$, and properties of Bochner integrals
$$
 \int_{-\infty}^T  ( \tilde u(t), w) D^{(\rho),*}_\infty v(t) dt = \Big( \int_0^T  (\Delta{\tilde  u}(t)+f(t)) v(t)dt,w\Big),
$$
where $(\cdot,\cdot)$ is the dual pairing of $H^1_0(\Omega)$. Then by the smoothness of $v$, the left hand side satisfies
$$
 \int_{-\infty}^T  ( \tilde u(t), w) D^{(\rho),*}_\infty v(t) dt = \int_{-\infty}^T  ( \tilde u(t)D^{(\rho),*}_\infty v(t) , w) dt  =  \Big( \int_{-\infty}^T \tilde u(t) D^{(\rho),*}_\infty v(t) \,dt, w\Big).
$$
Therefore, we derive that
$$
   \int_{-\infty}^T \tilde u(t) D^{(\rho),*}_\infty v(t) \,dt =  \int_0^T  (\Delta{\tilde u}(t)+f(t)) v(t)dt .
$$
This confirms that $\widetilde{D^{(\rho)}_\infty}\tilde u = \Delta{\tilde u} +f  \in L^2(0,T;H^{-1}(\Omega))$, and we proved that $u$ is a weak solution to problem  \eqref{preRL}.

\end{proof}

%\color{red}
%If $D_u:=\widetilde{D^{(\nu)}_\infty u}\in L^2(0,T;H^{-1}(\Omega)$, then $\int_0^T (D_u(t),\varphi_y)^2dt<\infty$ for all $\phi_y\in H^1_0$. Then by \cite{} for every $\phi_s\in ?$
%$$
%\int_0^T(D_u(t),\varphi_y)\varphi_s(t)\,dt=\left(\int_0^TD_u(t)\varphi_s(t)\,dt,\varphi_y\right),
%$$
%then, can one make sense of
%$$
%\left(\int_0^TD_u(t)\varphi_s(t)\,dt,\varphi_y\right)=\left(\int_0^Tu(t)D^{(\nu,*)}_\infty\varphi_s(t)\,dt,\varphi_y\right)\, ?
%$$
%Anyway we obtsain in the theorem that $\tilde u$ satisfy
%$$
 %\langle  {\widetilde {D_{\infty}^{(\rho)}}} \tilde u ,   v \rangle + \langle  \nabla \tilde u ,  \nabla  v \rangle = \langle f , v \rangle, \int_{0}^T \left(\int_\Omega u(t,x)\varphi_y(x) dx\right) \varphi_s(t) dt =\int_{0}^T \mathcal D(t),\varphi_y)\varphi_s(t)dt
%$$
%for $\mathcal D\in L^2(0,T;\H^{-1}$  as $\mathcal D = \tilde \Delta u....$

%\color{black}
 \begin{remark}\label{rem:unique}
 The uniqueness of the weak solution can be derived straightforwardly, provided that the kernel function is variables-separable and satisfies the assumption given in Lemma \ref{lem:density2}.
  In case that $f=0$ and $\phi=0$,  we let $\{\lambda_n\}_{n=1}^\infty$ be the set of all eigenvalues of $-\Delta$
  and $\varphi_n$ be the eigenfunction corresponding to $\lambda_n$. Then for any $\psi\in C_c^\infty(0,T)$, we have
 \begin{equation*}
  \int_0^T  {\widetilde {D_{\infty}^{(\rho)}}} (u(t), \fy_n) \psi(t)\,dt +  \int_0^T \lambda_n (u(t), \fy_n)\psi(t)\,dt = 0
\end{equation*}
As a result,  $(u(t), \fy_n)$ is the solution of the initial value problem
 \begin{equation*}
 {\widetilde {D_{\infty}^{(\rho)}}} (u(t), \fy_n) +  \lambda_n (u(t), \fy_n)  = 0\qquad \text{with} ~~(u(t), \fy_n) = 0~~\text{for all}~t<0.
\end{equation*}
Then Lemma \ref{lem:density2} and the uniqueness of the solution \cite[Section 3]{DYZ17}\footnote{The uniqueness argument for the initial value problem in \cite[Section 3]{DYZ17} can be extended to time-dependent kernels $\rho$ satisfying \ref{H0}.} yields that $(u(t), \fy_n) = 0$ for all $n$, and hence $u(t)\equiv0$. See \cite{allen17} for a discussion of uniqueness of weak solutions in the time-fractional case.
\end{remark}

\begin{remark}
If $\phi(t,x) \equiv \phi_0(x)\in H^1_0(\Omega)$ in Theorem \ref{thm:main},
then one recovers the weak solution to the (inhomogeneous) Caputo-type fractional diffusion equation \cite{Chen17,HKT17}
\begin{align*}
u(t,x)   =\mathbf E\left[\phi_0\left(B^x(\tau_0(t)\right)\mathbf1_{\{\tau_0(t) < \tau_\Omega(x)\}}\right] + \mathbf E\left[\int^{\tau_0(t) \wedge \tau_\Omega(x)}_0 f\left(-X^{t,(\rho)}(s),B^x(s)\right) ds \right].
\end{align*}
\end{remark}

\begin{remark}\label{rem:conti0}
The solution in Theorem \ref{thm:main} will be continuous at $t=0$ for every $x\in\Omega$ if $\phi$ is continuous at every point in $\{0\}\times\Omega$  and $\tau_0 :[0,T]\to \mathbb R$ is continuous. This can be proved by a stochastic continuity argument for the first term of the solution \eqref{thm:Caputo_SR}, and for the second term one can use $\mathbf E[\tau_0 (t)]\to 0$ as $t\downarrow 0$ (which is a consequence of the continuity of $\tau_0 $). However, the solution \eqref{thm:Caputo_SR} will in general fail to be continuous at $t=0$ even for smooth data. This is for example the case of integrable kernels $\int_0^\infty \rho(r)\,dr<\infty$ (see \cite[Remark A.3]{T18}). See Section \ref{sec:ik} for an example.
\end{remark}

%\begin{remark}
%If one wants to substitute in Theorem \ref{thm:main} the spatial operator $(\Delta_\Omega,\text{Dom}(\Delta_\Omega))$
%with a more general generator of a strongly continuous uniformly bounded semigroup on $C_{\partial\Omega}(\Omega)$,
%say $(\mathcal L_\Omega,\text{Dom}(\mathcal L_\Omega))$, it is sufficient to: (i) find an invariant core $\mathcal C_\Omega$ for
%$(\mathcal L_\Omega,\text{Dom}(\mathcal L_\Omega))$ so that $\mathcal L_\Omega $ allows a pointwise representation
%$\mathcal L_\Omega^{(\rho)} $ on  $\mathcal C_\Omega$, (ii) whose dual  $\mathcal L_\Omega^{(\rho),*} $
%maps $C^\infty_c((0,T)\times\Omega)$ in $L^1((-\infty,0)\times\Omega)\cap L^\infty((-\infty,0)\times\Omega)$ (essentially the respective version of Corollary \ref{cor:bddG-r2}), and (iii) find assumptions on the data to obtain the respective version of Lemma \ref{lem:density3}.
%\end{remark}

\section{Numerical Results}
In this section, we present some numerical results to verify the stochastic representation formula.
To this end, we consider the one-dimensional nonlocal diffusion problem \eqref{preRL} in the unit interval $\Omega=(-1,1)$.

\subsection{Non-integrable kernels}
We start with the case of non-integrable kernel function
\begin{equation}\label{eqn:ker1}
 \rho_\delta(r)=(1-\alpha)\delta^{\alpha-1}r^{-\alpha-1}\mathbf 1_{(0,\delta)}(r),
\end{equation}
with $\alpha\in(0,1)$ and the following data.
\begin{itemize}
\item[(a)] initial data $\phi(x,t)=e^{5t}(1+x)(1-x)^2x$ and zero source term $f\equiv0$;
\item[(b)] trivial initial data $\phi(x,t)=0$ and  source term $f=\sin(10t)(1-x)x\sin(\pi x)$.
\end{itemize}
The kernel function is proposed in this way in order to keep that
$  \int_0^\delta r \rho_\delta(r) \,dr =1 $
and hence the nonlocal operator recovers the infinitesimal first-order derivative as the nonlocal horizon diminishes.
The analytical property of the model has been extensively studied in \cite{DYZ17}.

The stochastic process generated by spatially second-order derivative (with zero boundary conditions),
which is well-known as the killed Brownian motion in the domain $\Omega=(-1,1)$,
 can be simply approximated by the lattice random walk.
Specifically, we divided the interval  $(-1,1)$ into $M$ small intervals, with the uniform mesh size $h=2/M$ and
grid points $x_j = jh-1$, $j=0,1,2,\ldots,M.$
Then in each time level, the particle standing in the grid points $x_j$ will randomly move to grid points $x_{j-1}$ or $x_{j+1}$.
In case that the particle hits the boundary of $\Omega$, then the time is set as $\tau_\Omega(x_j)$.
Here we let $B_h^{x_j}(t)$ be the position where the particle starting at position $x_j$ arrives at time $t$.

Similarly, the stochastic process generated by the operator
$$-D_\delta^{(\rho)}  u (t) = -\int_0^\delta (u(t)-u(t-r)) \rho_\delta(r) \,dr $$
with historical initial data could also be approximated by a one-dimensional lattice random walk,
where the trajectory of the particle involves some long-distance jumps.
To numerically simulate the stochastic process, we discretize  $[0,T]$ into $N$ small intervals $[t_{n-1},t_n]$ with $n=1,2,\ldots,N$ and let $k=T/N$.
Then we consider the discretization (assume that $\delta=mk$)
\begin{equation}\label{eqn:disc}
\begin{split}
 D_\delta^{(\rho)} u(t_n) %= \int_0^\delta (u(t_n)-u(t_n-r))\rho_\delta(r)\,dr\\
=&\ \int_0^k (u(t_n)-u(t_n-r)) \rho_\delta(r)\,dr + \sum_{j=2}^m \int_{(j-1)k}^{jk} (u(t_n)-u(t_n-r)) \rho_\delta(r) \,dr \\
\approx &\ \frac{u(t_n)-u(t_{n-1})}{k}\int_0^k  r\rho_\delta(r)\,ds +  \sum_{j=2}^m (u(t_n)-u(t_{n-k}))\int_{(j-1)k}^{jk}\rho_\delta(r) \,dr \\
= &\ \frac1{k^\alpha}\Big(\omega_0 u(t_n)- \sum_{j=1}^m \omega_j u(t_{n-j}) \Big)=: \bar D_\delta^{(\rho)} u(t_n). \\
\end{split}
\end{equation}
Here the weights $\{\omega_j\}_{j=0}^m$ are computed exactly as
$$ \omega_0=\delta^{\alpha-1}\Big(1+\frac{1-\alpha}{\alpha} (1-m^{-\alpha})\Big),\quad \omega_1=\delta^{\alpha-1}  $$
and
$$\omega_j=\delta^{\alpha-1}\frac{1-\alpha}{\alpha}((j-1)^{-\alpha} - j^{-\alpha}), ~~j=2,3,\ldots,m.  $$
%Here we note that
%$$ \sum_{i=1}^m \omega_i = \omega_0.$$
At each time level, the particle standing at the grid point $t_j$ will jump to one of the grid points $t_{j-i}$, for $i=1,2,...,m$, with the probability
$ p_i = \omega_i / \omega_0 $. It is easy to verify that $\sum_{j=1}^m \omega_j = \omega_0$ and hence $\sum_{j=1}^m p_j =1$.
We let $\tau_0(t_n)$ be the time that the particle starting at $t_n$ passes $0$, and $X_k^{t_n,(\rho)}(\tau_0(t_n))$  be the position where the particle arrives below $0$.
Then by applying the scaling $2\alpha k^\alpha = h^2 \delta^{\alpha-1}$, the solution of the nonlocal-in-time evolution equation \eqref{preRL} can be computed by
\begin{equation*}
\begin{split}
  U_h^n =
&\mathbf E\left[ \phi\left(-X_k^{t_n,(\rho)}(\tau_0(t)),B_h^{x_j}(\tau_0(t_n))\right)\mathbf 1_{\{\tau_0(t_n)< \tau_\Omega(x_j)\}}\right]\\
&\quad+\mathbf E\left[\int_0^{\tau_0(t_n)\wedge \tau_\Omega(x_j)}f\left(-X_k^{t_n,(\rho)}(s),B_h^{x_j}(s)\right)ds\right],
\end{split}
\end{equation*}
using the Monte Carlo method, where the integral of the second term is computed by the trapezoid rule.

\begin{figure}[h!]
\centering
\subfigure[$T=0.1$]{\includegraphics[trim = 0cm 0cm 0cm .2cm, clip=true,height=4cm]{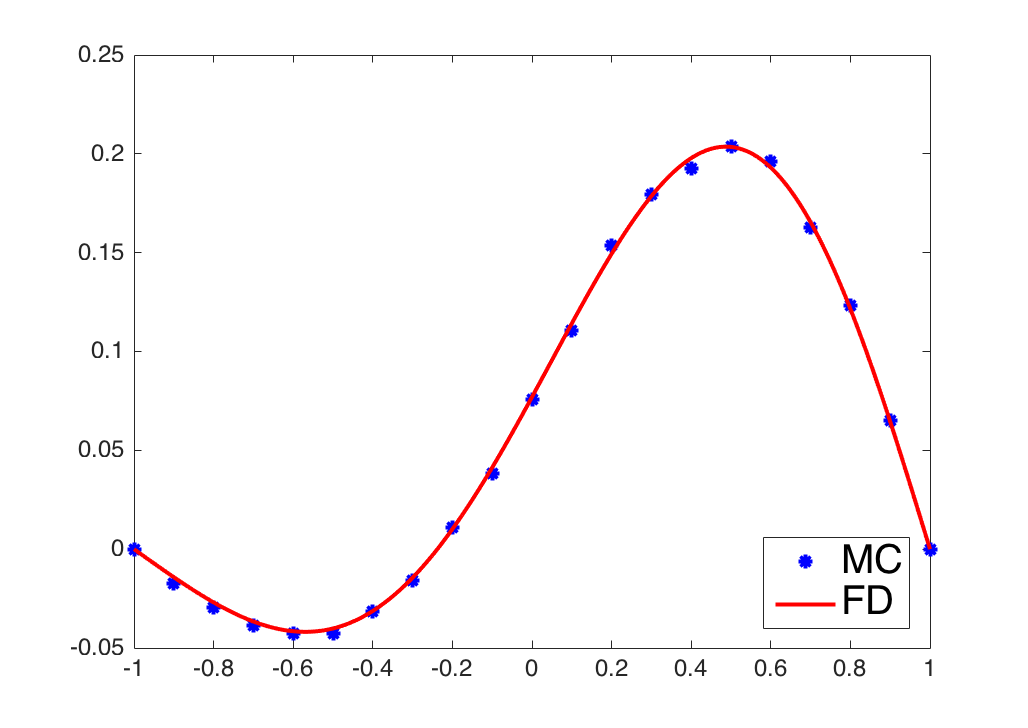}}
\subfigure[$T=0.2$]{\includegraphics[trim = 0cm 0cm 0cm .2cm, clip=true,height=4cm]{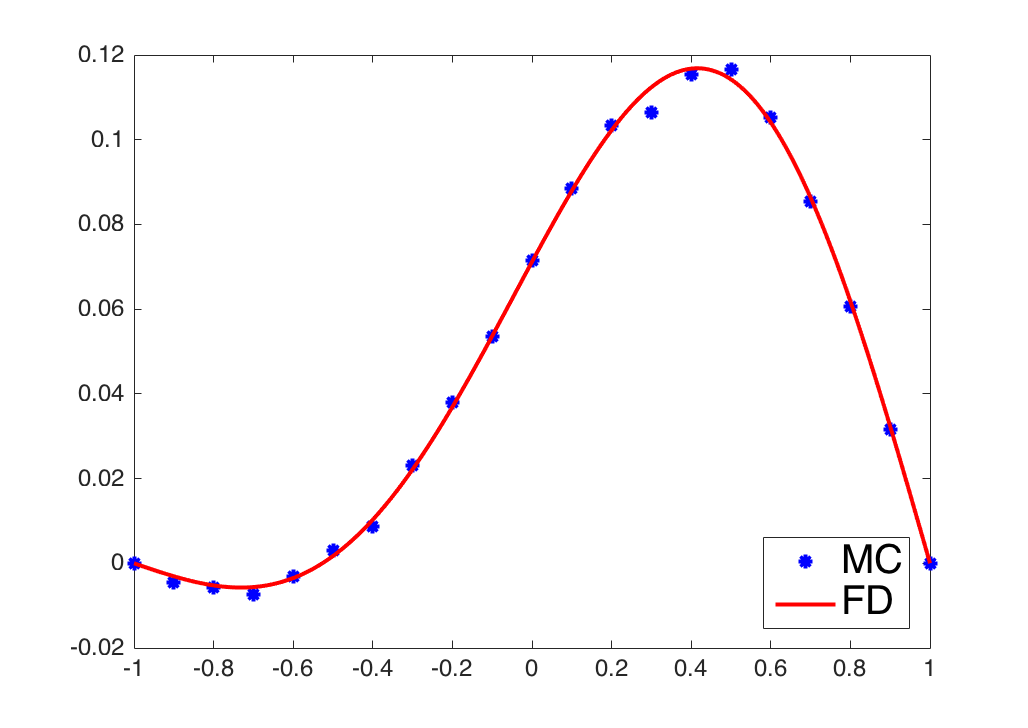}}\\
\subfigure[$T=0.4$]{\includegraphics[trim = 0cm 0cm 0cm .2cm, clip=true,height=4cm]{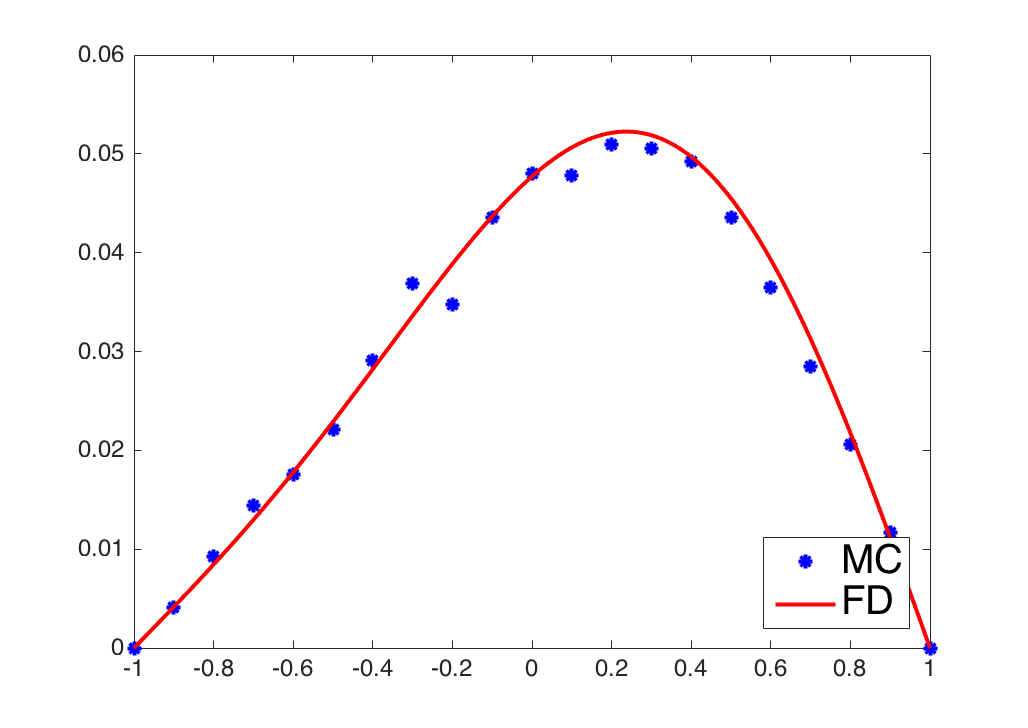}}
\subfigure[$T=0.6$]{\includegraphics[trim = 0cm 0cm 0cm .2cm, clip=true,height=4cm]{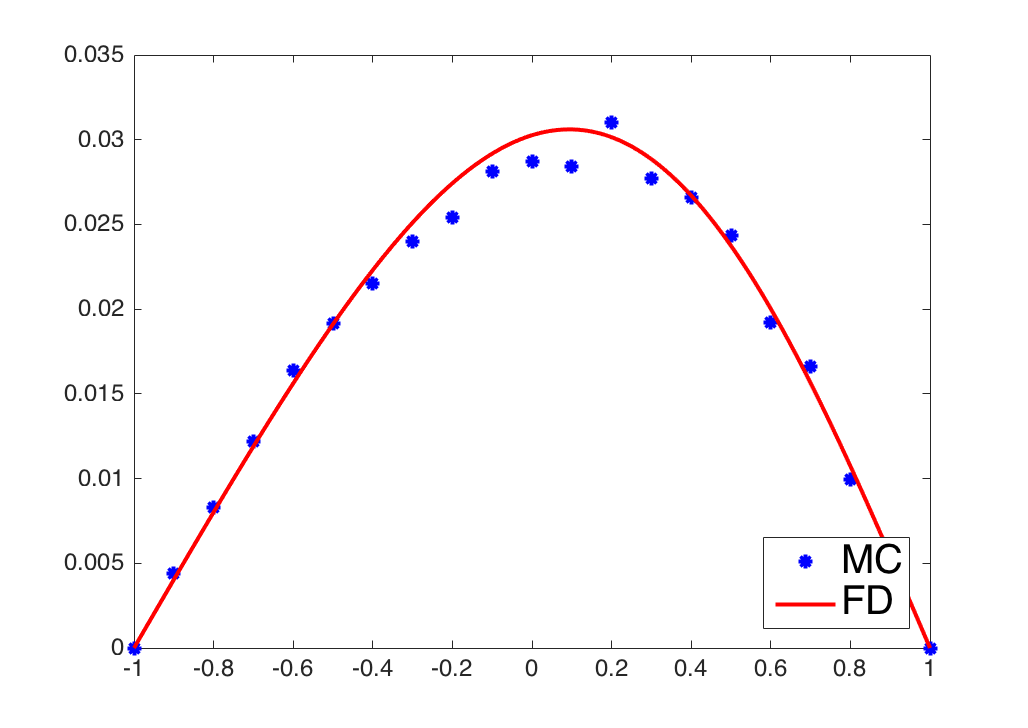}}
 \caption{Numerical solutions of Example (a) with $\delta=0.2$ and $\alpha=0.75$. (Blue dots: numerical solutions computed by the stochastic representation and Monte Carlo method (MC), with
  $h=0.02$, $k={^\alpha\hskip-7pt\sqrt{h^2\delta^{\alpha-1}/2\alpha}}$ and $50000$. Red curves: reference solutions computed by finite difference method (FD) with $h=10^{-3}$ and $k=10^{-4}$.)}\label{fig:a}
\end{figure}

In Figures \ref{fig:a} and \ref{fig:b}, we plot the numerical solution of nonlocal-in-time diffusion model \eqref{preRL}
where the nonlocal operator involves the finite-horizon kernel function \eqref{eqn:ker1} with $\alpha=0.75$ and $\delta=0.2$, at different time levels, $T=0.1$, $0.2$, $0.4$ and $0.6$.
To compute the numerical solution, we let $h=0.02$ and   $k={^\alpha\hskip-7pt\sqrt{h^2\delta^{\alpha-1}/2\alpha}} $ , and
use $50000$ Monte Carlo trials.
Since the closed form of the analytical solution is not available, the benchmark solutions
are computed by finite difference scheme
\begin{equation*}\label{master3}
\begin{split}
\bar D_\delta^{(\rho)} u_h^n  - \bar\partial_{xx}^h u_h^n = f^n
\end{split}
\end{equation*}
with a very fine mesh, say $k=10^{-4}$ and $h=10^{-3}$, where the discrete operator in time $\bar D_\delta^{(\rho)} $ is given by  \eqref{eqn:disc}
and the spatial one $\bar\partial_{xx}^h$ is the central difference approximation to the second order derivative.
In Figures \ref{fig:a} and \ref{fig:b}, the solution computed using the stochastic representation
 formula and the Monte Carlo method (MC) is plotted by blue dots while the finite difference solution (FD) is plotted by the red curves.
 We observe that the numerical solution computed by the stochastic approach
is very close to the one computed by the finite difference scheme, which supports our theoretical results.

\begin{figure}[h!]
\centering
\subfigure[$T=0.1$]{\includegraphics[trim = 0cm 0cm 0cm .2cm, clip=true,height=4cm]{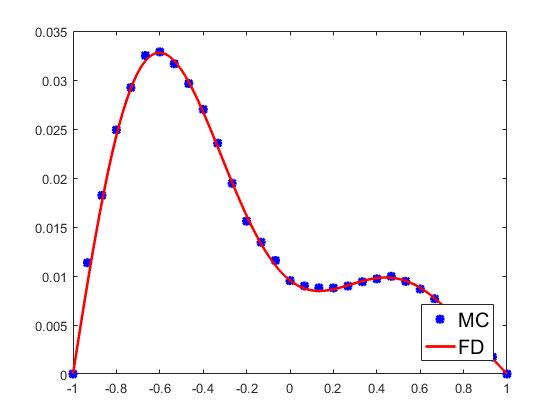}}
\subfigure[$T=0.2$]{\includegraphics[trim = 0cm 0cm 0cm .2cm, clip=true,height=4cm]{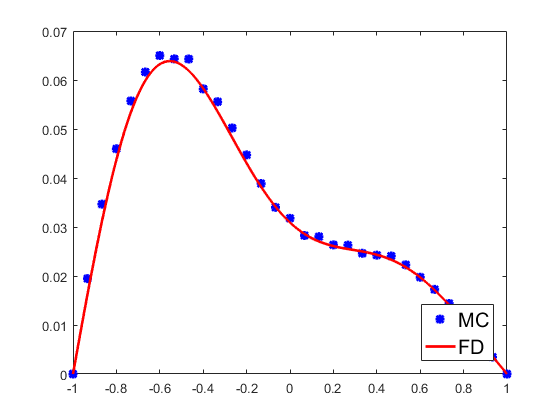}}\\
\subfigure[$T=0.4$]{\includegraphics[trim = 0cm 0cm 0cm .2cm, clip=true,height=4cm]{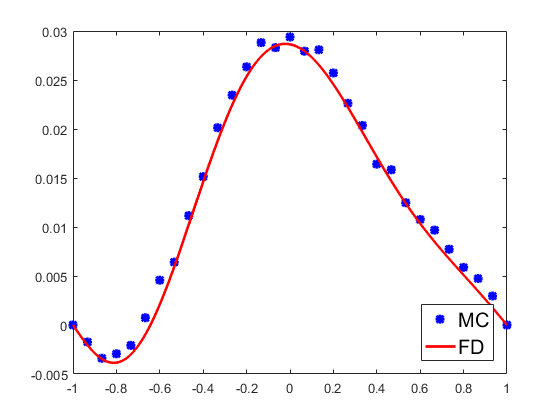}}
\subfigure[$T=0.6$]{\includegraphics[trim = 0cm 0cm 0cm .2cm, clip=true,height=4cm]{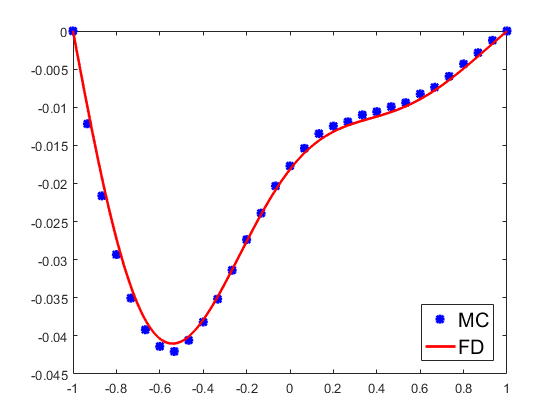}}
 \caption{Numerical solutions of Example (b) with $\delta=0.2$ and $\alpha=0.75$. (Blue dots: numerical solutions computed by the stochastic representation and Monte Carlo method (MC), with
  $h=0.02$, $k={^\alpha\hskip-7pt\sqrt{h^2\delta^{\alpha-1}/2\alpha}}$ and $50000$ trials. Red curves: reference solutions computed by finite difference method (FD) with $h=10^{-3}$ and $k=10^{-4}$.)}\label{fig:b}
\end{figure}

\subsection{Integrable kernels}\label{sec:ik}
Next, we present some numerical results for a special integrable kernel
which is the Dirac measure concentrated at $\delta>0$ weighted by $\lambda>0$, i.e.,
$$-D_{\infty}^{(\rho)}u(t):=(u(t-\delta)-u(t))\lambda.$$
 This nonlocal operator is the generator of a decreasing Poisson process, which performs negative jumps of size $\delta$ after   a $\lambda$-exponential waiting time. Hence we have
 $$t-X^{(\rho)}(\tau_0(t))=t-n\delta\quad \text{a.s.},\quad \text{for} ~~t\in((n-1)\delta,n\delta],~~n\in\mathbb N,$$
and $\tau_0(t)$ is a $\text{Gamma}(n,\lambda)$ random variable. Then solution to problem \eqref{preRL} with zero source term $f=0$
 allows the stochastic representation \eqref{thm:Caputo_SR} for $t\in((n-1)\delta,n\delta]$, $n\in\mathbb N$,
\begin{equation*}
u(t,x)=\mathbf E\left[\phi\left( t-n\delta, B^x(\tau_0(t))\right)\mathbf 1_{\{\tau_0(t)<\tau_\Omega(x)\}}\right].%=\int_\Omega \phi( t-n\delta, y)\left(\int_0^\infty p^\Omega_s(x,y)p^{\tau_0(t)}(s)ds\right)dy.
\end{equation*}
Note that even if $\phi\in C^\infty([-\delta,0]\times\overline\Omega)$, in general
$$\lim_{t\downarrow 0} u(t,x)=\mathbf E[\phi( -\delta, B^x(\tau_0(1)))\mathbf 1_{\{\tau_0(1)<\tau_\Omega(x)\}}]\neq \phi(0,x).$$
In Figure \ref{c}, we plot the numerical solutions (blue dots) with $\lambda=1$ at different time levels, where $h=0.04$ and  $50000$ Monte Carlo trials are used.
Again,  the reference solutions, plotted by red curves,  are computed by the finite difference method
$$ \lambda(u_h^n - u_h^{n-\delta/k}) - \bar \partial_{xx} u_h^n = f^n,\quad n=1,2,\ldots,N, $$
with very fine meshes, i.e., $k=10^{-3} $ and $h=10^{-3}$.
Numerical results show that the Monte Carlo simulation using the Feynman-Kac formula approximates the solution very well.
\begin{figure}[h!]
\centering
\subfigure[$T=0.1$]{\includegraphics[trim = 0cm 0cm 0cm .2cm, clip=true,height=4cm]{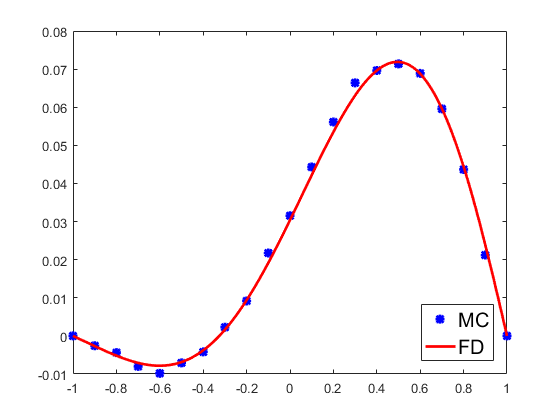}}
\subfigure[$T=0.25$]{\includegraphics[trim = 0cm 0cm 0cm .2cm, clip=true,height=4cm]{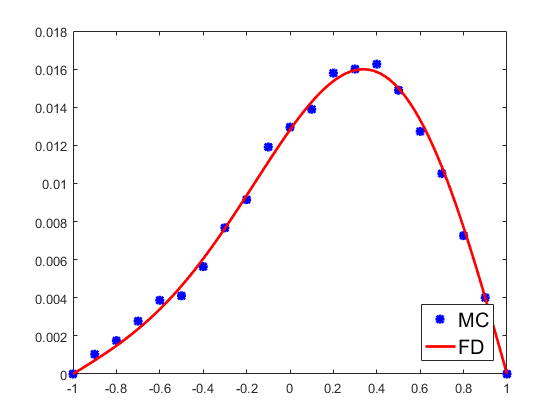}} \\
\subfigure[$T=0.3$]{\includegraphics[trim = 0cm 0cm 0cm .2cm, clip=true,height=4cm]{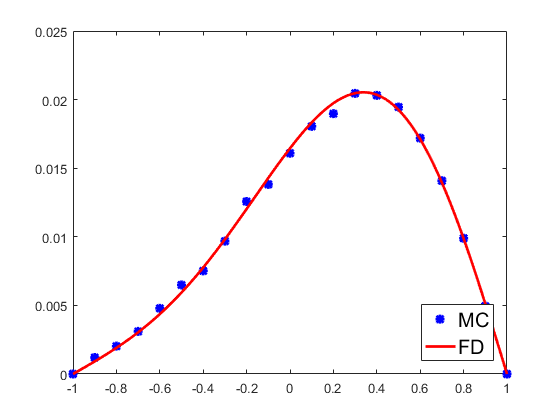}}
\subfigure[$T=0.45$]{\includegraphics[trim = 0cm 0cm 0cm .2cm, clip=true,height=4cm]{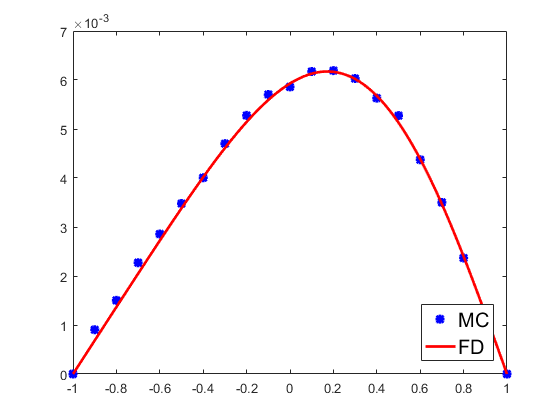}}
 \caption{Numerical solutions for the integrable kernel with $\lambda=1$ and $\delta=0.2$. (Blue dots: numerical solutions computed by MC, with
  $h=0.04$ and $50000$ trials. Red curves: reference solutions computed by FD with $h=10^{-3}$ and $k=10^{-3}$.)}
\end{figure}\label{c}

\section{Concluding Remarks}
In this paper, we study the stochastic representation for an initial-boundary value problem of a nonlocal-in-time evolution
equation \eqref{preRL}, where the nonlocal operator appearing in the model is the Markovian generator of a $(-\infty,T]$-valued decreasing L\'evy-type  process.
Under certain hypothesis, we derive the Feynman-Kac formula of the solution
by reformulating the original problem into a Caputo-type nonlocal model with a specific forcing term.
The case of weak data is also studied by energy arguments.
The stochastic representation leads to
a numerical scheme based on the Monte Carlo approach
for both integrable or non-integrable kernel functions.
The current theoretical results could be used to  give more rigorous analysis of the stochastic algorithms  for the nonlocal-in-time model.
It is also an interesting topic to study some quantitative  properties, such as asymptotical compatibility with shrinking nonlocal horizon parameter, of those algorithms.

%\section*{Acknowledgement}
%

%\bibliographystyle{imsart-nameyear}

\end{document}